\documentclass[11pt, leqno]{article}
\usepackage{amsmath, amsfonts, amsthm}
 \numberwithin{equation}{section}
\newtheorem{theorem}{Theorem}[section]

\begin{document}

\author{Ajai Choudhry}
\title{Matrix Morphology and  Composition \\of Higher Degree Forms\\ with Applications to  Diophantine Equations}
\date{}
\maketitle

\begin{abstract} 
In this paper we use matrices, whose entries satisfy certain linear conditions, to obtain composition identities $f(x_i)f(y_i)=f(z_i)$, where $f(x_i)$ is an irreducible form, with integer coefficients, of degree $n$ in $n$ variables ($n$ being $3,\,4,\,6$ or $8$), and $x_i,\,y_i,\;i=1,\,2,\,\ldots,\,n$, are independent variables while  the values of $z_i,\;i=1,\,2,\,\ldots,\,n$, are given by bilinear forms in the variables $x_i,\,y_i$.  When $n=2,\,4$ or $8$, we also obtain   composition identities $f(x_i)f(y_i)f(z_i)=f(w_i)$ where, as before, $f(x_i)$ is an irreducible form, with integer coefficients, of degree $n$ in $n$ variables while  
$x_i,\,y_i,z_i,\;i=1,\,2,\,\ldots,\,n$, are independent variables and the values of $w_i,\;i=1,\,2,\,\ldots,\,n$, are given by trilinear forms in the variables $x_i,\,y_i,\,z_i$,  and such that the identities cannot be derived from any identities of the type  $f(x_i)f(y_i)=f(z_i)$. Further, we  describe a method of obtaining  both these types of  composition identities for forms of higher degrees. The composition identities given in this paper have not been obtained earlier. We also  obtain infinitely many solutions in positive integers of certain quartic and octic diophantine equations $f(x_1,\,\ldots,\,x_n)=1$ where $f(x_1,\,\ldots,\,x_n)$ is a form that admits a composition identity and $n=4$ or $8$.
\end{abstract}


 
\section{Introduction}
A form $f(x_1,\,x_2,\,\ldots,\,x_n)$   in $n$ variables $x_i,\;i=1,\,2,\,\ldots,\,n$,  is said to be a form admitting composition if there exists an identity,
\begin{equation}
 f(x_1,\,x_2,\,\ldots,\,x_n)f(y_1,\,y_2,\,\ldots,\,y_n)=f(z_1,\,z_2,\,\ldots,\,z_n),\label{complaw}
\end{equation}
where  the variables $z_i,\;i=1,\,2,\,\ldots,\,n$, are given by bilinear forms in the variables $x_i,\,y_i, \;i=1,\,2,\,\ldots,\,n$, that is,
\begin{equation}
z_i=\sum_{j=1}^n\sum_{k=1}^n\lambda_{ijk}x_jy_k, \quad i=1,\,2,\,\ldots,\,n, \label{valz}
\end{equation}
where $\lambda_{ijk}$ are certain constants. When an identity \eqref{complaw} exists, we will say that $f(x_1,\,x_2,\,\ldots,\,x_n)$ is a composable form. 

The subject of higher degree forms  that admit composition has been studied by several authors \cite{Di, Ha, MD, Pu, Sch1, Sch2}. Dickson \cite[pp.~222, 224]{Di} has given general theorems describing all high degree  ternary and quaternary forms admitting composition. While these theorems yield  higher degree composable forms with complex coefficients, they are of little help in finding higher degree forms that admit composition and have only  integer coefficients.
It seems that the existing literature contains only two explicit nontrivial examples of high degree composable forms with integer coefficients, namely the determinant of an $n \times n $ matrix yields a composable form of degree $n$ in $n^2$ variables,  and the norm of an algebraic integer yields a composable form of degree $n$ in $n$ variables. 

In this paper we will  study forms admitting composition with a view to obtaining infinitely many integer solutions of certain higher degree diophantine equations, and accordingly we will  consider only those  forms which have  integer coefficients and such that the constants $\lambda_{ijk}$ in the relations \eqref{valz} are also all  integers. Further, when we refer to a form being irreducible, we mean irreducibility over $\mathbb{Q}$. 

We will use matrices, whose entries satisfy certain linear conditions, to obtain composition identities \eqref{complaw} when  $f(x_1,\,x_2,\,\ldots,\,x_n)$ is an irreducible form of degree $n$ in $n$ variables and $n=3,\,4,\,6$ or $8$. The forms obtained in this paper are not the norms of algebraic integers, and the  composition identities satisfied by these forms have not been obtained earlier. 

We also obtain  forms  that satisfy a law of composition that is a variation of the usual  composition law defined by \eqref{complaw}. We say that a form $f(x_1,\,x_2,\,\ldots,\,x_n)$ admits  three-fold composition if   there exists an identity
\begin{multline}
f(x_1,\,x_2,\,\ldots,\,x_n)f(y_1,\,y_2,\,\ldots,\,y_n)f(z_1,\,z_2,\,\ldots,\,z_n)\\
= f(w_1,\,w_2,\,\ldots,\,w_n),\quad \quad \quad \label{f3compform}
\end{multline}
where the values of $w_i, \;i=1,\,2,\,\ldots,\,n$, are given by  trilinear forms in the variables $x_i,\,y_i,\,z_i, \;i=1,\,2,\,\ldots,\,n$, and further, the identity \eqref{f3compform} cannot be derived from an identity of type \eqref{complaw}. 

We will show that every  binary quadratic form admits three-fold composition. We also  obtain examples of quaternary quartic  forms  and octonary octic forms that admit three-fold composition. It seems that  three-fold composition of forms has not been considered in the existing literature and all these results are new. 

We have used the composition identities \eqref{complaw} and \eqref{f3compform} to solve certain quartic and higher degree  diophantine equations,
\begin{equation}
f(x_1,\,x_2,\,\ldots,\,x_n)=1.
\end{equation}
We give several examples of such equations for which we obtain, by an iterative process, infinitely many solutions in positive integers.

In Section~\ref{mm} we discuss how matrices, with a certain structure, may be used to obtain forms that satisfy composition identities. In Section~\ref{highdegfms} we construct examples of cubic and higher degree forms admitting composition and we also solve certain diophantine equations related to these forms. In Section~\ref{threefold} we construct forms  admitting three-fold composition and consider diophantine equations related to such forms. We conclude the paper with certain remarks and open problems regarding matrix morphology and composition of forms as well as related higher degree diophantine equations.

\section{Matrix morphology and  composition of forms}\label{mm}
If $A=[a_{ij}]$ and $B=[b_{ij}]$ are two arbitrary $n \times n$ matrices in independent  variables $x_{ij},\,y_{ij},\; i=1,\,\ldots,\,n,\;j=1,\,\ldots,\,n$,  and $AB=C$, then, as already observed in the Introduction,  the identity
\begin{equation}
{\rm det}(A) \times {\rm det}(B) ={\rm det}(C), \label{detidentgen}
\end{equation}
immediately yields  the  composable form ${\rm det}(A)$ of degree $n$ in $n^2$ variables. The form ${\rm det}(A)$ is, however,  of little interest as the number of variables is too large compared to the degree of the form. We will show that if the entries of the matrices $A$ and $B$ satisfy certain linear conditions, the identity \eqref{detidentgen} can be used to obtain composable forms of degree $n$ in just $n$ variables for several values of $n \geq 3$. 

\subsection{Matrices with a linear structure}\label{mmls}
 If $A=[a_{ij}]$ is an $n \times  n$  square matrix whose entries satisfy  $k$  independent linear conditions,
\begin{equation}
\sum_{i=1}^n \sum_{j=1}^n \lambda_{hij}a_{ij}=0,\quad h=1,\,2,\,\ldots,\,k, \label{linset}
\end{equation}
where $\lambda_{hij}$ are constants and  $k \; < n^2-1$,  we will say that the matrix $A$ has a linear structure defined by Eqs.~\eqref{linset}. 
It is evident that the entries of any scalar multiple of the matrix $A$ also satisfy similar  linear conditions. If $B=[b_{ij}]$ is any other  $n \times  n$ square matrix with the same linear structure as that of the matrix $A$, that is,   the entries $b_{ij}$ of the matrix $B$   satisfy the linear conditions,
\begin{equation}
\sum_{i=1}^n \sum_{j=1}^n \lambda_{hij}b_{ij}=0,\quad h=1,\,2,\,\ldots,\,k, \label{linsetB}
\end{equation} 
then it is readily seen that the entries of  the matrix $A+B$  also satisfy similar linear conditions. Thus, the linear structure is preserved under the operations of  scalar multiplication and addition  of matrices.

If there exist two matrices $A$ and $B$ with the same linear structure, defined by the relations \eqref{linset} and \eqref{linsetB}, and the entries of the matrix product $AB$, denoted by  the matrix $C=[c_{ij}]$, also satisfy similar linear conditions, that is, 
\begin{equation}
\sum_{i=1}^n \sum_{j=1}^n \lambda_{hij}c_{ij}=0,\quad h=1,\,2,\,\ldots,\,k, \label{linsetC}
\end{equation}
we say that  the linear structure of the matrix  $A$ is preserved under multiplication. 

The existing literature contains a few  examples of such  matrices,  for instance, it has been noted that the linear structure of circulant matrices,  semi-magic matrices and centrosymmetric matrices is preserved under multiplication (\cite[Theorem 3.2.4, p. 74]{Da}, \cite{We}, \cite{We2}). It seems, however,  that till now the morphology of matrices has not been used to obtain explicit composition identities for cubic and higher degree forms. 

We note that if $A=[a_{ij}]$ is a matrix with a linear structure defined by the conditions \eqref{linset}, then the entries $a_{ij}$ of the matrix $A$ may be written in terms of linear forms in $n^2-k$ independent variables. Conversely, it is readily seen that if the entries $a_{ij}$ of an $n \times n$  matrix are given by linear forms in  $h$ independent  variables, then the entries $a_{ij}$ satisfy $n^2-h$ linear conditions. We may thus specify the linear structure of a matrix just  by giving the entries of the matrix as linear forms in a certain number of variables. 

When the entries of a matrix  $A(x_1,\,x_2,\,\ldots,\,x_h)$ are expressed as linear  forms in the independent variables $x_1,\,x_2,\,\ldots,\,x_h$, we may replace the variables $x_1,\,x_2,\,\ldots,\,x_h$ by a new set of independent variables $y_1,\,y_2,\,\ldots,\,y_h$ and immediately obtain a second matrix $A(y_1,\,y_2,\,\ldots,\,y_h)$ with the same linear structure. If the linear structure of the matrix $A(x_1,\,x_2,\,\ldots,\,x_h)$ is preserved under multiplication, we may write,
\begin{equation}
A(x_1,\,x_2,\,\ldots,\,x_h) \times A(y_1,\,y_2,\,\ldots,\,y_h)=A(z_1,\,z_2,\,\ldots,\,z_h), \label{Axyz}
\end{equation}
where the values of $z_i,\; i=1,\,2,\,\ldots,\,h$,   given by bilinear forms in the variables $x_i,\, y_i, \;  i=1,\,2,\,\ldots,\,h$.

If we write $f(x_1,\,x_2,\,\ldots,\,x_h)={\rm det}\left(A(x_1,\,x_2,\,\ldots,\,x_h)\right)$, then from the relation \eqref{Axyz}, we immediately obtain the composition identity,
\begin{equation}
f(x_1,\,x_2,\,\ldots,\,x_h)f(y_1,\,y_2,\,\ldots,\,y_h)=f(z_1,\,z_2,\,\ldots,\,z_h), \label{fxyz}
\end{equation}
with  the values of $z_i,\; i=1,\,2,\,\ldots,\,h$,  being given by bilinear forms in the variables $x_i,\, y_i, \;  i=1,\,2,\,\ldots,\,h$.

As a simple illustrative example, if we define the 
matrix $A(x_1,\,x_2)$ by
\begin{equation}
A(x_1,\,x_2)=\begin{bmatrix} x_1 & x_2 \\ -qx_2 & x_1+px_2\end{bmatrix}, \label{22ex2}
\end{equation}
where $p, q $ are   arbitrary integers, it is  readily verified that
\begin{equation}
A(x_1,\,x_2) \times A(y_1,\,y_2)=A(z_1,\,z_2), \label{identA}
\end{equation}
where
\begin{equation}
z_1 =x_1y_1 -qx_2y_2, \quad z_2 =x_1y_2 + x_2y_1+ px_2y_2. \label{valzex2}
\end{equation}
Thus the linear structure of the matrix $A$ defined by \eqref{22ex2} is  preserved under multiplication, and we  get the well-known composition of forms identity, 
\begin{equation}
(x_1^2+px_1x_2+qx_2^2)(y_1^2+py_1y_2+qy_2^2) = z_1^2+pz_1z_2+qz_2^2, \label{compqdfm}
\end{equation}
where the values of $z_i,\;i=1,\,2$, are given by \eqref{valzex2}.

We also note that if $M$ is the companion matrix of the   polynomial $x^n+a_1x^{n-1}+a_2x^{n-2}+\cdots+a_{n-1}x+a_n$, and $I$ is the $n \times n $ identity matrix, the linear structure of the matrix $x_1I+x_2M+x_3M^2+\cdots +x_nM^{n-1}$ is preserved under multiplication, and we thus get the composition formulae for forms that are norms of  algebraic integers.

 In the sections that follow  we will construct new examples of $n \times n $ matrices in $n$ variables such that the linear structure of the matrices is preserved under multiplication, and we thereby get new results on composition of forms of degrees 3, 4, 6 and 8. 

\subsection{ A $3 \times 3 $ matrix with a linear structure}\label{mm33}

We now  define a $3 \times 3$ matrix $A(x_1,\,x_2,\,x_3)$, whose entries are linear forms in 3 independent variables $x_1,\,x_2,\,x_3$, as follows: 
\begin{equation}
		A(x_1,\,x_2,\,x_3)=\begin{bmatrix} x_{11}& x_{12}& x_{13} \\ x_{21} & x_{22} & x_{23}\\ 
		x_{31}  & x_{32} & x_{33}  \end{bmatrix}, \label{cubicmatrix}\\	
 \end{equation}
where 
\begin{equation}
\begin{aligned}
x_{11}&=x_1, \quad x_{12}=x_2,\quad x_{13}=x_3,\\
x_{21}&=- \lambda_3( \lambda_1- \lambda_2- \lambda_3+ \lambda_5)x_2- \lambda_3( \lambda_2- \lambda_4)x_3 ,\\
x_{22}&=x_1+ \lambda_1x_2+ \lambda_2x_3, \\
x_{23} & = \lambda_3x_2+ \lambda_3x_3, \\
x_{31}&=- \lambda_3( \lambda_2- \lambda_4)x_2+(- \lambda_1 \lambda_4+ \lambda_2^2- \lambda_2 \lambda_5+ \lambda_3 \lambda_4)x_3,\\
x_{32}&=\lambda_2x_2+ \lambda_4x_3, \\
x_{33} &= x_1+ \lambda_3x_2+ \lambda_5x_3, \label{xijcubicmatrix}
\end{aligned}
\end{equation}
and $\lambda_i,\;i=1,\,2\,\ldots,\,5$, are arbitrary integers.
It is readily verified by direct computation that
\begin{equation}
A(x_1,\,x_2,\,x_3) \times A(y_1,\,y_2,\,y_3)=A(z_1,\,z_2,\,z_3), \label{mm33ident}
\end{equation}
where the values of $z_i,\;i=1,\,2,\,3$, are given by bilinear forms in the variables $x_i,\,y_i$ as follows:
\begin{equation}
\begin{aligned}
z_1&=  x_1  y_1- \lambda_3 ( \lambda_1 - \lambda_2- \lambda_3+ \lambda_5)  x_2  y_2- \lambda_3 ( \lambda_2- \lambda_4)  x_2  y_3 \\
& \quad \quad - \lambda_3 ( \lambda_2- \lambda_4)  x_3  y_2+(- \lambda_1   \lambda_4+ \lambda_2^2- \lambda_2  \lambda_5+ \lambda_3  \lambda_4)  x_3  y_3,\\
 z_2& =  x_1  y_2+ x_2  y_1+ \lambda_1   x_2  y_2+ \lambda_2  x_2  y_3+ \lambda_2  x_3  y_2+ \lambda_4  x_3  y_3, \\
z_3& =  x_1  y_3+ \lambda_3  x_2  y_2+ \lambda_3  x_2  y_3+ x_3  y_1+ \lambda_3  x_3  y_2+ \lambda_5  x_3  y_3.
\end{aligned}
\label{valzmm33}
\end{equation}

Thus the matrix $A$ has a linear structure that is preserved under multiplication.

\subsection{Higher order matrices with a linear structure}\label{mmhom}
We will now show that given two square matrices of orders $m$ and $n$ such that the individual linear structures of the two matrices are preserved under multiplication, we can  construct  a square  matrix of order $mn$  with a linear structure that is preserved under multiplication. 

Let $U=[u_{ij}]$ be an $n \times n$ matrix with a linear structure that is preserved under multiplication. As noted above, the entries $u_{ij}$  of the matrix $U$ may be written as linear forms in $h$ independent variables where $h$ is a positive integer  $< n^2$, that is, we may write,
\begin{equation}
u_{ij}=\sum_{r=1}^h \lambda_{ijr}a_r, \quad i=1,\,2,\ldots,\,n,\;j=1,\,2,\ldots,\,n, \label{valuij}
\end{equation}
where $\lambda_{ijr}$ are integers while $a_r,\;r=1,\,2,\ldots,\,h$, are arbitrary  variables. 

In the entries of the matrix $U$, we now replace the variables $a_r,\;r=1,\,2,\ldots,\,h$, by a set of new  variables $b_r,\;r=1,\,2,\ldots,\,h$, and thus construct a new matrix $V=[v_{ij}]$ which has the same linear structure as the matrix $U$ and whose entries $v_{ij}$ are given by
\begin{equation}
v_{ij}=\sum_{r=1}^h \lambda_{ijr}b_r, \quad i=1,\,2,\ldots,\,n,\;j=1,\,2,\ldots,\,n. \label{valvij}
\end{equation}

Since, by assumption, the linear structure of the matrix $U$ is preserved under multiplication, if we write  $UV=W$, the entries of the matrix $W=[w_{ij}]$ may be written as
 \begin{equation}
w_{ij}=\sum_{r=1}^h \lambda_{ijr}c_r, \quad i=1,\,2,\ldots,\,n,\;j=1,\,2,\ldots,\,n. \label{valwij}
\end{equation}
where the values of the variables $c_r,\;r=1,\,2,\ldots,\,h$, are given by bilinear forms in the variables $a_1,\,a_2,\ldots,\,a_h$ and $b_1,\,b_2,\ldots,\,b_h$, that is, 
\begin{equation}
c_r=\sum_{s=1}^h \sum_{t=1}^h \mu_{rst}a_sb_t, \quad r=1,\,2,\ldots,\,h. \label{valw}
\end{equation}

We have assumed in the beginning that we are also given a second square matrix of order $m$ whose linear structure is preserved under multiplication, and hence its entries are expressible as linear forms in a certain number of variables. By successively replacing these variables by new sets of variables, we  can readily generate $2h$ square matrices  $A_i,\;i=1,\,2,\,\ldots,\,h$, and $B_i,\;i=1,\,2,\,\ldots,\,h$, of order $m$ such that all the matrices $A_i,\;B_i$ have  an identical linear structure  that is preserved under multiplication.

We will now use the matrices $U,\; V$ and the $2h$ matrices  $A_i,\;B_i$ to construct two new matrices $P$ and $Q$ whose entries will be written as certain matrices. We will construct the matrix $P$ by replacing the $h$ variables $a_i$ in the entries $u_{ij}$ of the matrix $U$ by the $h$ matrices $A_i$ respectively, and we similarly construct the matrix $Q$ by replacing the $h$ variables $b_i$  in the entries $v_{ij}$ of the matrix $V$ by the $h$ matrices $B_i$. We thus get,
\begin{equation}
P=\begin{bmatrix} P_{11} & P_{12} & \ldots & P_{1n} \\ P_{21} & P_{22} & \ldots & P_{2n} \\ \vdots \\  P_{n1} & P_{n2} & \ldots & P_{nn}
\end{bmatrix}, \quad 
Q=\begin{bmatrix} Q_{11} & Q_{12} & \ldots & Q_{1n} \\ Q_{21} & Q_{22} & \ldots & Q_{2n} \\ \vdots \\  Q_{n1} & Q_{n2} & \ldots & Q_{nn}
\end{bmatrix},
\end{equation}
where $P_{ij}$ and $Q_{ij}$ are matrices given by
\begin{align}
P_{ij}&=\sum_{r=1}^h \lambda_{ijr}A_r, \quad i=1,\,2,\ldots,\,n,\;j=1,\,2,\ldots,\,n, \label{valpij}\\
Q_{ij}&=\sum_{r=1}^h \lambda_{ijr}B_r, \quad i=1,\,2,\ldots,\,n,\;j=1,\,2,\ldots,\,n. \label{valqij}
\end{align}
It is clear from the manner of construction that all the matrices $P_{ij}$ and $Q_{ij}$ have the same linear structure as the matrices $A_i$ and $B_i$ and further, the linear structure of the matrix  $P$ is the same as that of the matrix $Q$.  

We will now show that the linear structure of the matrices $P$ and $Q$ is preserved under multiplication. The matrices $P$ and $Q$ are naturally conformally partitioned, and if we write  $PQ=R$, in view of the relations \eqref{valuij}, \eqref{valvij}, \eqref{valwij}, \eqref{valpij} and \eqref{valqij}, we may write the matrix $R$ as
\begin{equation}
R=\begin{bmatrix} R_{11} & R_{12} & \ldots & R_{1n} \\ R_{21} & R_{22} & \ldots & R_{2n} \\ \vdots \\  R_{n1} & R_{n2} & \ldots & R_{nn}
\end{bmatrix},
\end{equation}
where $R_{ij}$ is the matrix given by
\begin{equation}
R_{ij}=\sum_{r=1}^h \lambda_{ijr}C_r, \quad i=1,\,2,\ldots,\,n,\;j=1,\,2,\ldots,\,n. \label{valrij}
\end{equation}
and  the values of the matrices $C_r,\;r=1,\,2,\ldots,\,h$, are given by  
\begin{equation}
C_r=\sum_{s=1}^h \sum_{t=1}^h \mu_{rst}A_sB_t, \quad r=1,\,2,\ldots,\,h. \label{valC}
\end{equation}

Since the matrices $A_s$ and $B_t$ have an identical  linear structure for all values of $s$ and $t$, and this linear structure is preserved both under addition and multiplication of matrices, it follows that the $h$ matrices $C_r,\;i=1,\,2,\,\ldots,\,h$, have  the same linear structure as the matrices $A_i$ and $B_i$. Hence, in view of the relations \eqref{valrij},  all the  matrices $R_{ij}$ also have the same linear structure which is identical with the linear structure of  the matrices $P_{ij}$ and $Q_{ij}$.

Apart from the linear structure of the matrices $P_{ij}$ themselves,  the entries of the matrix $P$ satisfy certain additional linear conditions since the values of the $n^2$ matrices $P_{ij}$, as  given by \eqref{valpij}, are expressed  as a linear combination of  $h$ independent matrices $A_r,\;r=1,\,2,\ldots,\,h$. We note that the values of the  matrices $Q_{ij}$ and $R_{ij}$  given by \eqref{valqij} and \eqref{valrij} are expressed by exactly similar linear combinations of independent matrices $B_r, \;i=1,\,2,\,\ldots,\,h$, and $C_r, \;i=1,\,2,\,\ldots,\,h$ respectively, and therefore the additional linear conditions satisfied by the entries of the matrix $P$ are also satisfied by the entries of the matrices $Q$ and $R$. It now follows that the linear structure of the matrix $PQ=R$ is the same as that of the matrices $P$ and $Q$. 

Since the matrices $U$ and $V$ are square matrices of order $n$ and the matrices $A_i,\,B_i$ are square matrices of order $m$, it follows that $P$ and $Q$ are square matrices of order $mn$. Thus, given two square matrices of orders $m$ and $n$ with linear structures that are preserved under multiplication, we have  constructed a square matrix of order $mn$ with a linear structure that is also preserved under multiplication.

\section{Higher degree forms admitting composition and related diophantine equations}\label{highdegfms}
In this section we  will   construct forms of degree $n$ in $n$ variables when $n=3,\,4,\,6$ and $8$ such that these forms admit a composition identity \eqref{complaw}. We note that these forms cannot be expressed as the norms of any algebraic integers, and our results have not been obtained earlier.

We will use the composition identities \eqref{complaw}  to obtain infinitely many solutions in positive integers of certain diophantine equations,
\begin{equation}
f(x_1,\,x_2,\,\ldots,\,x_n)=1, \label{diopheqngen}
\end{equation}
where $f(x_1,\,x_2,\,\ldots,\,x_n)$ is a form of degree $n$ in $n$ variables and $n=4 $ or  $8$.

We also obtain infinitely many solutions in positive integers of the diophantine equation \eqref{diopheqngen} when $f(x_1,\,x_2,\,\ldots,\,x_n)$ is a quartic form in $6$  variables, and we show that these solutions can neither  be obtained by any parametric solution of Eq.~\eqref{diopheqngen} nor can they be derived from the integer points of any curve of genus 1. 

\subsection{Cubic forms}\label{cubicfms} In Section~\ref{mm33} we have already obtained a $3 \times 3$ matrix $A(x_1,\,x_2,\,x_3)$ with a linear structure that is preserved under multiplication. It now follows from  \eqref{mm33ident} that $f(x_1,\,x_2,\,x_3)={\rm det}(A(x_1,\,x_2,\,x_3)$ is a ternary cubic form admitting composition. Direct computation gives the form,
\begin{multline}
f(x_1,\,x_2,\,x_3)=x_1^3+(\lambda_1 +\lambda_3 )x_1^2x_2+(\lambda_2 +\lambda_5 )x_1^2x_3\\
+\lambda_3 (2\lambda_1 -2\lambda_2 -\lambda_3 +\lambda_5 )x_1x_2^2
+(\lambda_1 \lambda_5 +2\lambda_2 \lambda_3 -3\lambda_3 \lambda_4 )x_1x_2x_3\\
+(\lambda_1 \lambda_4 -\lambda_2 ^2+2\lambda_2 \lambda_5 -2\lambda_3 \lambda_4 )x_1x_3^2+\lambda_3^2(\lambda_1 -2\lambda_2 -\lambda_3 +\lambda_4 +\lambda_5 )x_2^3\\
-\lambda_3 (2\lambda_1 \lambda_4 -\lambda_1 \lambda_5 -2\lambda_2 ^2-\lambda_2 \lambda_3 +3\lambda_2 \lambda_5 -\lambda_3 \lambda_4 +\lambda_3 \lambda_5 -\lambda_5 ^2)x_2^2x_3\\
+(\lambda_1 ^2\lambda_4 -\lambda_1 \lambda_2 ^2+\lambda_1 \lambda_2 \lambda_5 -3\lambda_1 \lambda_3 \lambda_4 +\lambda_2 ^2\lambda_3 +\lambda_2 \lambda_3 \lambda_4 \\
+2\lambda_3 ^2\lambda_4 -2\lambda_3 \lambda_4 \lambda_5 )x_2x_3^2
+(\lambda_1 \lambda_2 \lambda_4 -\lambda_2 ^3+\lambda_2 ^2\lambda_5 -2\lambda_2 \lambda_3 \lambda_4 +\lambda_3 \lambda_4 ^2)x_3^3,
\end{multline}
and we have the composition identity,
\begin{equation}
f(x_1,\,x_2,\,x_3)f(y_1,\,y_2,\,y_3)=f(z_1,\,z_2,\,z_3), \label{cubicfmident}
\end{equation}
where $x_i,\,y_i,\;i=1,\,2,\,3$, are independent variables and the values of $z_i\;i=1,\,2,\,3$, are given by \eqref{valzmm33}.

We note that if $\alpha$ is a root of a monic cubic equation, the coefficients of $x_1^3,\,x_2^3,\,x_3^3$ in the norm of the  algebraic integer $x_1+x_2\alpha+x_3\alpha^2$ are in geometric progression. Since the coefficients of $x_1^3,\,x_2^3,\,x_3^3$ in the form $f(x_1,\,x_2,\,x_3)$ are not in geometric progression, it follows that  $f(x_1,\,x_2,\,x_3)$ is not the norm of an algebraic integer. 

\subsection{Quartic forms}\label{quarticfms}
In Section~\ref{quarticfms1} we will  obtain a quaternary quartic form  admitting composition and in Section~\ref{quarticfms2} we will consider a related quartic diophantine equation.

\subsubsection{}\label{quarticfms1} We will  now construct a $4 \times 4$ matrix with a linear structure that is preserved under multiplication. We will follow the method, as well as the notation, of Section~\ref{mmhom} to construct the desired matrix.

We first choose the matrix $U$ as follows:
\begin{equation}
U=\begin{bmatrix}
a_1 & a_2 \\ -qa_2 & a_1+pa_2\end{bmatrix}. \label{valU}
\end{equation}
The entries of the matrix $U$ are linear forms in the variables $a_1,\,a_2$ while $p, q$ are arbitrary integers. Further,  the  the linear structure of the matrix $U$ is exactly the same as that of the matrix $A(x_1,\,x_2)$ defined by \eqref{22ex2} and is thus preserved under multiplication. 

We will now construct the matrix $P$ by replacing the variables $a_1$ and $a_2$ in the matrix $U$ by matrices $A_1$ and $A_2$ both of which will have an identical linear structure that is preserved under multiplication. We thus get the matrix $P$ as
\begin{equation}
P=\begin{bmatrix}
A_1 & A_2 \\ -qA_2 & A_1+pA_2\end{bmatrix}. \label{valP}
\end{equation}

We will choose the matrices $A_1$ and $A_2$ as follows:
\begin{equation}
A_1=A(x_1,\,x_2), \quad A_2=A(x_3,\,x_4),
\end{equation}
where 
\begin{equation}
A(x_1,\,x_2)=\begin{bmatrix}
x_1 & x_2 \\ -nx_2 & x_1+mx_2\end{bmatrix},
\end{equation}
where $m,\,n$ are arbitrary integers. We note that the linear structure of the matrix $A(x_1,\,x_2)$, and hence also that of the matrices $A_1$ and $A_2$, is preserved under multiplication, and hence the linear structure of the matrix $P$ will be preserved under multiplication. 

The entries of the matrix $P$ are now given by linear forms in the variables $x_1,\,x_2,\,x_3,\,x_4$, and we may write the matrix $P=P(x_1,\,x_2,\,x_3,\,x_4)$ as follows:
\begin{equation}
\begin{bmatrix}
x_1 & x_2 & x_3 & x_4 \\ -n x_2 & x_1+m x_2 & -n x_4 & x_3+m x_4 \\ -q x_3 & -q x_4 & x_1+p x_3 & x_2+p x_4 \\ q n x_4 & -q (x_3+m x_4) & -nx_2-p n x_4 & x_1+m x_2+p (x_3+m x_4)
\end{bmatrix}.
\label{m44}
\end{equation}

We now replace the variables $x_1,\,x_2,\,x_3,\,x_4$ in the matrix $P(x_1,x_2,x_3,x_4)$ by independent variables $y_1,\,y_2,\,y_3,\,y_4$ to get a new matrix $P(y_1,\,y_2,\,y_3,\,y_4)$ with the same linear structure. Since the linear structure of the two matrices $P(x_1,\,x_2,\,x_3,\,x_4)$ and $P(y_1,\,y_2,\,y_3,\,y_4)$ is identical and is preserved under multiplication, we get the relation,
\begin{equation}
P(x_1,\,x_2,\,x_3,\,x_4) \times P(y_1,\,y_2,\,y_3,\,y_4)=P(z_1,\,z_2,\,z_3,\,z_4), \label{quarticfmPxyz} 
\end{equation}
where the values of $z_i,\;i=1,\,\ldots,\,4$, obtained readily from the first row of the matrix product $P(x_1,\,x_2,\,x_3,\,x_4) \times P(y_1,\,y_2,\,y_3,\,y_4)$,  are given by
\begin{equation}
\begin{aligned}
z_1& = x_1y_1 - nx_2y_2 - qx_3y_3 + qnx_4y_4,\\
 z_2 &= x_1y_2 + x_2y_1 + mx_2y_2 - qx_3y_4 - qx_4y_3 - mqx_4y_4,\\
 z_3& = x_1y_3 - nx_2y_4 + x_3y_1 + px_3y_3 - nx_4y_2 - npx_4y_4, \\
z_4& = x_1y_4 + x_2y_3 + mx_2y_4 + x_3y_2 + px_3y_4 \\
& \quad \quad + x_4y_1 + mx_4y_2 + px_4y_3 + mpx_4y_4.
\end{aligned}
\label{valzquarticfm}
\end{equation}
 
As an immediate consequence of the relation \eqref{quarticfmPxyz}, if we now write, 
\[
f(x_1,\,x_2,\,x_3,\,x_4) ={\rm det}(P(x_1,\,x_2,\,x_3,\,x_4)),\]
we get the composition identity,
\begin{equation}
f(x_1,\,x_2,\,x_3,\,x_4)f(y_1,\,y_2,\,y_3,\,y_4) =f(z_1,\,z_2,\,z_3,\,z_4), \label{complawquarticfm}
\end{equation}
where $f(x_1,\,x_2,\,x_3,\,x_4)$ is the quaternary quartic form given by 
\begin{multline}
f(x_1,\,x_2,\,x_3,\,x_4)=x_1^4 + 2mx_1^3x_2 + 2px_1^3x_3 + mpx_1^3x_4 \\
+ (m^2 + 2n)x_1^2x_2^2 + 3mpx_1^2x_2x_3 + (m^2 + 2n)px_1^2x_2x_4 \\
+ (p^2 + 2q)x_1^2x_3^2 + (p^2 + 2q)mx_1^2x_3x_4 + (m^2q + np^2 - 2nq)x_1^2x_4^2\\
 + 2mnx_1x_2^3 + (m^2 + 2n)px_1x_2^2x_3 + 3mnpx_1x_2^2x_4 + (p^2 + 2q)mx_1x_2x_3^2\\
 + (m^2p^2 + 8nq)x_1x_2x_3x_4 + (p^2 + 2q)mnx_1x_2x_4^2 + 2pqx_1x_3^3\\
 + 3mpqx_1x_3^2x_4 + (m^2 + 2n)pqx_1x_3x_4^2 + mnpqx_1x_4^3 + n^2x_2^4\\
 + mnpx_2^3x_3 + 2n^2px_2^3x_4 + (m^2q + np^2 - 2nq)x_2^2x_3^2 \\
+ (p^2 + 2q)mnx_2^2x_3x_4 + (p^2 + 2q)n^2x_2^2x_4^2 + mpqx_2x_3^3\\
 + (m^2 + 2n)pqx_2x_3^2x_4 + 3mnpqx_2x_3x_4^2 + 2n^2pqx_2x_4^3 + q^2x_3^4\\
 + 2mq^2x_3^3x_4 + (m^2 + 2n)q^2x_3^2x_4^2 + 2mnq^2x_3x_4^3 + n^2q^2x_4^4, \label{quarticfm}
\end{multline}
and the values of $z_i,\;i=1,\,\ldots,\,4$, are given by \eqref{valzquarticfm}.

As in the case of the ternary cubic form obtained in Section \ref{cubicfms}, it is readily observed that the form $f(x_1,\,x_2,\,x_3,\,x_4)$ is not the norm of an algebraic integer. 
 
\subsubsection{}\label{quarticfms2}

We will now consider the diophantine equation,
\begin{equation}
f(x_1,\,x_2,\,x_3,\,x_4)=1, \label{quarticeqn}
\end{equation}
where  $f(x_1,\,x_2,\,x_3,\,x_4)$ is the quartic form defined by \eqref{quarticfm}. We will first show that the set $S$ of integer solutions of Eq.~\eqref{quarticeqn} forms an abelian  group with respect to a suitably defined operation of addition. 

We note that $(1,0,0,0)$ is a solution of Eq.~\eqref{quarticeqn} and hence the set $S$ is nonempty. Next, we observe that if $(x_1, \, x_2, \,x_3,\, x_4)$ and $(y_1, \, y_2, \,y_3,\, y_4)$ are any two integer solutions of Eq.~\eqref{quarticeqn}, and the values of   $z_i,\;i=1,\,\ldots,\,4$, are defined by \eqref{valzquarticfm}, then $(z_1, \, z_2, \,z_3,\, z_4)$ is also an integer solution of Eq.~\eqref{quarticeqn}. Accordingly, we define the operation of addition as follows:
\begin{equation}
(x_1, \, x_2, \,x_3,\, x_4)+(y_1, \, y_2, \,y_3,\, y_4)=(z_1, \, z_2, \,z_3,\, z_4). \label{defaddition}
\end{equation}

It is readily seen that this operation is commutative, the identity element is $(1,0,0,0)$ and the operation is associative. The inverse $(y_1, \, y_2, \,y_3,\, y_4)$, of an arbitrary element $(x_1, \, x_2, \,x_3,\, x_4)$ of $S$,  is obtained by solving the equations \eqref{valzquarticfm} where we take $(z_1, \, z_2, \,z_3,\, z_4)=(1,0,0,0)$, and we thus obtain, 
\begin{equation}
\begin{aligned}
y_1&=x_1^3 + 2mx_1^2x_2 + 2px_1^2x_3 + mpx_1^2x_4 + (m^2 + n)x_1x_2^2 \\
& \quad \quad + 3mpx_1x_2x_3 + p(m^2 + 2n)x_1x_2x_4 + (p^2 + q)x_1x_3^2\\
& \quad \quad + m(p^2 + 2q)x_1x_3x_4 + (m^2q + np^2 - nq)x_1x_4^2 + mnx_2^3\\
& \quad \quad + m^2px_2^2x_3+ 2mnpx_2^2x_4 + mp^2x_2x_3^2 \\
 & \quad \quad+ (m^2p^2 + 2nq)x_2x_3x_4 + mn(p^2 + q)x_2x_4^2+ pqx_3^3 \\
 & \quad \quad+ 2mpqx_3^2x_4 + pq(m^2 + n)x_3x_4^2 + mnpqx_4^3, \\
y_2& = -x_1^2x_2 - mx_1x_2^2 - 2px_1x_2x_3 - mpx_1x_2x_4 - 2qx_1x_3x_4 \\
& \quad \quad- mqx_1x_4^2 - nx_2^3 - mpx_2^2x_3 - 2npx_2^2x_4\\
  & \quad \quad+ (-p^2 + q)x_2x_3^2- mp^2x_2x_3x_4 - n(p^2 + q)x_2x_4^2\\
  & \quad \quad	- pqx_3^2x_4 - mpqx_3x_4^2 - npqx_4^3,\\
 y_3& = -x_1^2x_3 - 2mx_1x_2x_3 - 2nx_1x_2x_4 - px_1x_3^2 - mpx_1x_3x_4 \\
& \quad \quad- npx_1x_4^2 + (-m^2 + n)x_2^2x_3 - mnx_2^2x_4 - mpx_2x_3^2\\
& \quad \quad - m^2px_2x_3x_4 - mnpx_2x_4^2 - qx_3^3 - 2mqx_3^2x_4 \\
& \quad \quad - q(m^2 + n)x_3x_4^2 - mnqx_4^3,\\
 y_4& = -x_1^2x_4 + 2x_1x_2x_3 + mx_2^2x_3 + nx_2^2x_4 + px_2x_3^2 \\
& \quad \quad+ mpx_2x_3x_4 + npx_2x_4^2 + qx_3^2x_4 + mqx_3x_4^2 + nqx_4^3.
\end{aligned}
\end{equation}

Since $(x_1, \, x_2, \,x_3,\, x_4)$ is assumed to be a solution of  Eq.~\eqref{quarticeqn}, if we now  substitute the above  values of $y_i$ in the identity \eqref{complawquarticfm}, we get $f(y_1, y_2, y_3, y_4)$ $=f(1,\,0,\,0,\,0)=1$ which confirms that  $(y_1,\,y_2,\,y_3,\,y_4)$ is indeed the inverse of $(x_1, \, x_2, \,x_3,\, x_4)$. We have thus established that the integer solutions of Eq.~\eqref{quarticeqn} form an abelian group.

We will now consider the diophantine equation $f(x_1,\,x_2,\,x_3,\,x_4)=1$ when $m=5,\,n=-23,\,p=2,\,q=-7$, that is, the equation,
\begin{multline}
x_1^4 + 10x_1^3x_2 + 4x_1^3x_3 + 10x_1^3x_4 - 21x_1^2x_2^2 + 30x_1^2x_2x_3\\
 - 42x_1^2x_2x_4 - 10x_1^2x_3^2 - 50x_1^2x_3x_4 - 589x_1^2x_4^2\\
 - 230x_1x_2^3 - 42x_1x_2^2x_3 - 690x_1x_2^2x_4 - 50x_1x_2x_3^2\\
 + 1388x_1x_2x_3x_4 + 1150x_1x_2x_4^2 - 28x_1x_3^3 - 210x_1x_3^2x_4\\
 + 294x_1x_3x_4^2 + 1610x_1x_4^3 + 529x_2^4 - 230x_2^3x_3 + 2116x_2^3x_4\\
 - 589x_2^2x_3^2 + 1150x_2^2x_3x_4 - 5290x_2^2x_4^2 - 70x_2x_3^3 \\
+ 294x_2x_3^2x_4 + 4830x_2x_3x_4^2 - 14812x_2x_4^3 + 49x_3^4 \\
+ 490x_3^3x_4 - 1029x_3^2x_4^2 - 11270x_3x_4^3 + 25921x_4^4=1. \label{quarticeqnex1}
\end{multline}  

It is readily verified that $(6,\,2,\,3,\,1)$ is a  numerical solution of \eqref{quarticeqnex1}. If $(\alpha_{11},\,\alpha_{12},\,\alpha_{13},\,\alpha_{14})$ is an arbitrary  integer solution of Eq.~\eqref{quarticeqnex1} such that $\alpha_{1i} > 0$ for each $i$, on adding $(6,\,2,\,3,\,1)$ to the solution $(\alpha_{11},\,\alpha_{12},\,\alpha_{13},\,\alpha_{14})$, we obtain a  solution $(\alpha_{21},\,\alpha_{22},\,\alpha_{23},\,\alpha_{24})$ of Eq.~\eqref{quarticeqnex1} where
\begin{equation}
\begin{aligned}
 \alpha_{21}&=6\alpha_{11} + 46\alpha_{12} + 21\alpha_{13} + 161\alpha_{14}, \\
 \alpha_{22}&=2\alpha_{11} + 16\alpha_{12} + 7\alpha_{13} + 56\alpha_{14}, \\
 \alpha_{23}&= 3\alpha_{11} + 23\alpha_{12} + 12\alpha_{13} + 92\alpha_{14}, \\
 \alpha_{24}&= \alpha_{11} + 8\alpha_{12} + 4\alpha_{13} + 32\alpha_{14}.
\end{aligned}
\end{equation}
Since $\alpha_{1i} > 0$ for each $i$, it immediately follows that $\alpha_{21} > \alpha_{11}$ and hence this solution is distinct from the solution $(\alpha_{11},\,\alpha_{12},\,\alpha_{13},\,\alpha_{14})$. Further, $\alpha_{2i} > 0$ for each $i$, and we may  therefore add $(6,\,2,\,3,\,1)$ to the solution $(\alpha_{21},\,\alpha_{22},\,\alpha_{23},\,\alpha_{24})$ to get a new solution $(\alpha_{31},\,\alpha_{32},\,\alpha_{33},\,\alpha_{34})$ such that $\alpha_{31} > \alpha_{21} > \alpha_{11}$, and $\alpha_{3i} > 0$ for each $i$,  and we may repeat the process any number of times to get an infinite sequence of integer solutions of Eq.~\eqref{quarticeqnex1}.

If we  take the initial  solution $(\alpha_{11},\,\alpha_{12},\,\alpha_{13},\,\alpha_{14})$ as $(6,\,2,\,3,\,1)$, the next three solutions of Eq.~\eqref{quarticeqnex1} obtained by the above process are as follows:
\[
\begin{array}{c}
(352,\, 121,\,  192,\, 66),\,\quad (22336,\,  7680,\,  12215,\,  4200),\,\\
 (1420011,\, 488257,\,  776628,\,  267036).
\end{array}\]

\subsection{Sextic forms}\label{sexticfms}
In Sections~\ref{sexticfms1} and \ref{sexticfms2},  we will  obtain two senary sextic forms  admitting composition and in Section~\ref{sexticfms3} we will consider  related diophantine equations.

\subsubsection{}\label{sexticfms1} We will  now construct a $6 \times 6$ matrix with a linear structure that is preserved under multiplication by exactly the same procedure as in Section~\ref{quarticfms1} except that in the matrix  \eqref{valU} we will now replace   the variables $a_1$ and $a_2$ by $3 \times 3$ matrices $A_1=A(x_1, x_2, x_3)$ and $A_2=A(x_4, x_5, x_6)$, where $A(x_1, x_2, x_3)$ is the $3 \times 3$  matrix defined by \eqref{cubicmatrix} and \eqref{xijcubicmatrix}. We thus get the $6 \times 6$ matrix
given below:
\begin{equation}
P(x_1,\, \ldots,\, x_6)=\begin{bmatrix} A(x_1, x_2, x_3) & A(x_4, x_5, x_6) \\
-qA(x_4, x_5, x_6) & A(x_1, x_2, x_3) + pA(x_4, x_5, x_6) \end{bmatrix}. \label{P66}
\end{equation}

Since the linear structure of the matrix $A(x_1,\,x_2, x_3)$ is preserved under multiplication, it follows that the linear structure of the matrix $P(x_1, \ldots, x_6)$ is also preserved under multiplication. If we now write
\begin{equation}
f(x_1, x_2, x_3, x_4, x_5, x_6)={\rm det}(\left(P(x_1,\,x_2,\,x_3,\,x_4,\,x_5,\,x_6)\right),
\end{equation}
we immediately get the composition formula,
\begin{equation}
f(x_1, x_2,\,\ldots, \,x_6)f( y_1,  y_2,\,\ldots, \,  y_6)=f( z_1,  z_2, \,\ldots, \,  z_6),
\end{equation}
where $f(x_1,\, x_2,\, \ldots,\, x_6)$ is a sextic form in the variables $x_i,\;i=1,\,\ldots,\,6$, while $\lambda_i,\;i=1,\,\ldots,\,5$, and $p, \,q$ are arbitrary integer parameters and the values of $z_i,\,i=1,\,\ldots,\,6$, are given by
\begin{equation}
\begin{aligned}
z_1 &= x_1y_1 - \lambda_3(\lambda_1 - \lambda_2 - \lambda_3 + \lambda_5)x_2y_2 - \lambda_3(\lambda_2 - \lambda_4)x_2y_3\\ & \quad
 - \lambda_3(\lambda_2 - \lambda_4)x_3y_2 + (-\lambda_1\lambda_4 + \lambda_2^2 - \lambda_2\lambda_5 + \lambda_3\lambda_4)x_3y_3\\ & \quad
 - qx_4y_4 + q\lambda_3(\lambda_1 - \lambda_2 - \lambda_3 + \lambda_5)x_5y_5 + q\lambda_3(\lambda_2 - \lambda_4)x_5y_6\\ & \quad
 + q\lambda_3(\lambda_2 - \lambda_4)x_6y_5 + q(\lambda_1\lambda_4 - \lambda_2^2 + \lambda_2\lambda_5 - \lambda_3\lambda_4)x_6y_6, \\ 
z_2 &= x_1y_2 + x_2y_1 + \lambda_1x_2y_2 + \lambda_2x_2y_3 + \lambda_2x_3y_2 + \lambda_4x_3y_3\\ & \quad
 - qx_4y_5 - qx_5y_4 - \lambda_1qx_5y_5 - \lambda_2qx_5y_6 - \lambda_2qx_6y_5 - \lambda_4qx_6y_6, \\
z_3 &= x_1y_3 + \lambda_3x_2y_2 + \lambda_3x_2y_3 + x_3y_1 + \lambda_3x_3y_2 + \lambda_5x_3y_3\\ & \quad
 - qx_4y_6 - q\lambda_3x_5y_5 - q\lambda_3x_5y_6 - qx_6y_4 - q\lambda_3x_6y_5 - \lambda_5qx_6y_6, \\
z_4 &= x_1y_4 - \lambda_3(\lambda_1 - \lambda_2 - \lambda_3 + \lambda_5)x_2y_5 - \lambda_3(\lambda_2 - \lambda_4)x_2y_6\\ & \quad
 - \lambda_3(\lambda_2 - \lambda_4)x_3y_5 + (-\lambda_1\lambda_4 + \lambda_2^2 - \lambda_2\lambda_5 + \lambda_3\lambda_4)x_3y_6\\ & \quad
 + x_4y_1 + px_4y_4 - \lambda_3(\lambda_1 - \lambda_2 - \lambda_3 + \lambda_5)x_5y_2 \\ & \quad
- \lambda_3(\lambda_2 - \lambda_4)x_5y_3 - \lambda_3(\lambda_1 - \lambda_2 - \lambda_3 + \lambda_5)px_5y_5 \\ & \quad
- \lambda_3(\lambda_2 - \lambda_4)px_5y_6 - \lambda_3(\lambda_2 - \lambda_4)x_6y_2 \\ & \quad
+ (-\lambda_1\lambda_4 + \lambda_2^2 - \lambda_2\lambda_5 + \lambda_3\lambda_4)x_6y_3 - \lambda_3(\lambda_2 - \lambda_4)px_6y_5\\ & \quad
 - p(\lambda_1\lambda_4 - \lambda_2^2 + \lambda_2\lambda_5 - \lambda_3\lambda_4)x_6y_6,\\ 
 z_5 &= x_1y_5 + x_2y_4 + \lambda_1x_2y_5 + \lambda_2x_2y_6 + \lambda_2x_3y_5 + \lambda_4x_3y_6\\ & \quad
 + x_4y_2 + px_4y_5 + x_5y_1 + \lambda_1x_5y_2 + \lambda_2x_5y_3 + px_5y_4 \\& \quad
+ \lambda_1px_5y_5 + \lambda_2px_5y_6 + \lambda_2x_6y_2 + \lambda_4x_6y_3 + \lambda_2px_6y_5 + \lambda_4px_6y_6, \\ 
z_6 &= x_1y_6 + \lambda_3x_2y_5 + \lambda_3x_2y_6 + x_3y_4 + \lambda_3x_3y_5 + \lambda_5x_3y_6 \\ & \quad
+ x_4y_3 + px_4y_6 + \lambda_3x_5y_2 + \lambda_3x_5y_3 + \lambda_3px_5y_5 + \lambda_3px_5y_6\\ & \quad
 + x_6y_1 + \lambda_3x_6y_2 + \lambda_5x_6y_3 + px_6y_4 + \lambda_3px_6y_5 + \lambda_5px_6y_6.
\end{aligned}
\label{sexticz}
\end{equation}

It has been verified, using the software MAPLE, that the sextic form $f(x_1, x_2,\,\ldots,\, x_6)$ is irreducible for arbitrary values of the parameters $\lambda_i,\,p$ and $q$ and further, the  form $f(x_1, x_2, x_3, x_4, x_5, x_6)$ is not the norm of an algebraic integer.

We do not give the sextic form $f(x_1, x_2, x_3, x_4, x_5, x_6)$ explicitly as it is too cumbersome to write. According to the software MAPLE, there are $11926$  terms in the expansion of $f(x_1, x_2, x_3, x_4, x_5, x_6)$.

As in the case of the diophantine equation \eqref{quarticeqn}, it is readily established that the set of integer solutions of the sextic diophantine equation $f(x_1, x_2, x_3, x_4, x_5, x_6)=1$ forms an abelian  group.

\subsubsection{}\label{sexticfms2} 
We will now obtain a senary sextic form admitting composition by taking  the matrices $A(x_1, x_2, x_3)$ and $A(x_4, x_5, x_6)$ in the matrix \eqref{P66} as  $3 \times 3$ circulant matrices. 
 For simplicity, we take $p=0$ and replace $q$ by -$q$, and now we get the matrix $P(x_1, x_2, x_3, x_4, x_5, x_6)$ may be written as 
\begin{equation}
P(x_1,\,\ldots,\, x_6)
=\begin{bmatrix}
x_1 &  x_2 &  x_3 &  x_4 &  x_5 &  x_6\\
x_3 &  x_1 &  x_2 &  x_6 &  x_4 &  x_5\\
x_2 &  x_3 &  x_1 &  x_5 &  x_6 &  x_4\\
qx_4   &   qx_5   &   qx_6   &   x_1   &   x_2   &   x_3\\
qx_6   &   qx_4   &   qx_5   &   x_3   &   x_1   &   x_2\\
qx_5   &   qx_6   &   qx_4   &   x_2   &   x_3   &   x_1
\end{bmatrix}.
\end{equation}
Since the linear structure of circulant matrices is preserved under multiplication, it follows that the linear structure of the matrix 
$P(x_1, x_2, x_3, x_4, x_5, x_6)$ is also preserved under multiplication. Thus, 
\[f(x_1, x_2, x_3, x_4, x_5, x_6)={\rm det}\left(P(x_1,\,x_2,\,x_3,\,x_4,\,x_5,\,x_6)\right),\]
gives a senary sextic form admitting composition.

We note that the above senary form has two irreducible factors given by
\begin{equation}
f(x_1, x_2,\,\ldots,\, x_6) = f_1(x_1, x_2,\,\ldots,\, x_6)f_2(x_1, x_2,\,\ldots,\, x_6),
\end{equation}
where
\begin{equation}
f_1(x_1, x_2, x_3, x_4, x_5, x_6)=(x_1+x_2+x_3)^2-q(x_4 + x_5 + x_6)^2,
\end{equation}
\begin{multline}
f_2(x_1, x_2, x_3, x_4, x_5, x_6)=x_1^4-(2x_2 +2x_3)x_1^3 + (3x_2^2 + 3x_3^2 - 2qx_4^2\\
 + 2qx_4x_5 + 2qx_4x_6 + qx_5^2 - 4qx_5x_6 + qx_6^2)x_1^2 \\
+ (-2x_2^3 + 2qx_2x_4^2 - 8qx_2x_4x_5 + 4qx_2x_4x_6 + 2qx_2x_5^2 + 4qx_2x_5x_6\\
 - 4qx_2x_6^2 - 2x_3^3 + 2qx_3x_4^2 + 4qx_3x_4x_5 - 8qx_3x_4x_6 - 4qx_3x_5^2\\
 + 4qx_3x_5x_6 + 2qx_3x_6^2)x_1+x_2^4-2x_2^3x_3+3x_2^2x_3^2+qx_2^2x_4^2\\
+2qx_2^2x_4x_5-4qx_2^2x_4x_6-2qx_2^2x_5^2+2qx_2^2x_5x_6+qx_2^2x_6^2-2x_2x_3^3\\
-4qx_2x_3x_4^2+4qx_2x_3x_4x_5+4qx_2x_3x_4x_6+2qx_2x_3x_5^2-8qx_2x_3x_5x_6\\
+2qx_2x_3x_6^2+x_3^4+qx_3^2x_4^2-4qx_3^2x_4x_5+2qx_3^2x_4x_6+qx_3^2x_5^2\\
+2qx_3^2x_5x_6-2qx_3^2x_6^2+q^2x_4^4-2q^2x_4^3x_5-2q^2x_4^3x_6+3q^2x_4^2x_5^2\\
+3q^2x_4^2x_6^2-2q^2x_4x_5^3-2q^2x_4x_6^3+q^2x_5^4-2q^2x_5^3x_6\\
+3q^2x_5^2x_6^2-2q^2x_5x_6^3+q^2x_6^4,
\end{multline}
and, in accordance with a theorem of Dickson \cite[p. 219, Theorem 3]{Di},  we now get the simultaneous composition  identities,
\begin{equation}
\begin{aligned}
f_1(x_1, x_2,\,\ldots, \,x_6)f_1( y_1,  y_2,\,\ldots, \,  y_6)=f_1( z_1,  z_2, \,\ldots, \,  z_6),\\
f_2(x_1, x_2,\,\ldots, \,x_6)f_2( y_1,  y_2,\,\ldots, \,  y_6)=f_2( z_1,  z_2, \,\ldots, \,  z_6),
\end{aligned}
\label{simulcompsextic}
\end{equation}
where $x_i,\, y_i,\; i=1,\,2,\,\ldots,\,6$, are arbitrary while the values of $z_i,\; i=1,\,2,\,\ldots,\,6$,  are given by
\begin{equation}
\begin{aligned}
z_1 &= x_1y_1 + x_2y_3 + x_3y_2 + qx_4y_4 + qx_5y_6 + qx_6y_5,\\
 z_2 &= x_1y_2 + x_2y_1 + x_3y_3 + qx_4y_5 + qx_5y_4 + qx_6y_6,\\
 z_3 &= x_1y_3 + x_2y_2 + x_3y_1 + qx_4y_6 + qx_5y_5 + qx_6y_4,\\
 z_4 &= x_1y_4 + x_2y_6 + x_3y_5 + x_4y_1 + x_5y_3 + x_6y_2,\\
 z_5 &= x_1y_5 + x_2y_4 + x_3y_6 + x_4y_2 + x_5y_1 + x_6y_3, \\
z_6 &= x_1y_6 + x_2y_5 + x_3y_4 + x_4y_3 + x_5y_2 + x_6y_1.
\end{aligned}
\label{valzsimulcompsextic}
\end{equation}

We note that the form $f_1(x_1, x_2, x_3, x_4, x_5, x_6)$ has just two independent variables, and accordingly we make a suitable linear transformation after which we can rewrite the formulae \eqref{simulcompsextic} for simultaneous composition of forms as follows:
\begin{equation}
\begin{aligned}
f_1(u_i)f_1( v_i)&=f_1( w_i),\\
f_2(u_i)f_2( v_i)&=f_2( w_i),
\end{aligned}
\label{simulcompsexticuvw}
\end{equation}
where
\begin{equation}
f_1(u_i)=u_1^2 - qu_2^2, \label{deff1sextic}
\end{equation}
and
\begin{multline}
f_2(u_i)=u_1^4-(6 u_3+6 u_6) u_1^3+(q u_2^2-6 q u_2 u_5+15 u_3^2+24 u_3 u_6\\
-3 q u_4^2+6 q u_4 u_5+6 q u_5^2+15 u_6^2) u_1^2+(-6 q u_2^2 u_6-12 q u_2 u_3 u_4\\
+12 q u_2 u_3 u_5+12 q u_2 u_4 u_6+24 q u_2 u_5 u_6-18 u_3^3-36 u_3^2 u_6\\
+18 q u_3 u_4^2-18 q u_3 u_5^2-36 u_3 u_6^2-36 q u_4 u_5 u_6-18 q u_5^2 u_6-18 u_6^3) u_1\\
+q^2u_2^4 - 6q^2u_2^3u_4 - 6q^2u_2^3u_5 + 15q^2u_2^2u_4^2 + 24q^2u_2^2u_4u_5\\
 + 15q^2u_2^2u_5^2 - 18q^2u_2u_4^3 - 36q^2u_2u_4^2u_5 - 36q^2u_2u_4u_5^2 \\
- 18q^2u_2u_5^3 + 9q^2u_4^4 + 18q^2u_4^3u_5 + 27q^2u_4^2u_5^2 + 18q^2u_4u_5^3\\
 + 9q^2u_5^4 - 3qu_2^2u_3^2 + 6qu_2^2u_3u_6 + 6qu_2^2u_6^2 + 18qu_2u_3^2u_4\\
 - 36qu_2u_3u_5u_6 - 18qu_2u_4u_6^2 - 18qu_2u_5u_6^2 - 18qu_3^2u_4^2\\
 - 18qu_3^2u_4u_5 + 9qu_3^2u_5^2 - 18qu_3u_4^2u_6 + 36qu_3u_4u_5u_6 \\
+ 36qu_3u_5^2u_6+ 9qu_4^2u_6^2 + 36qu_4u_5u_6^2 + 9qu_5^2u_6^2 + 9u_3^4\\
 + 18u_3^3u_6 + 27u_3^2u_6^2 + 18u_3u_6^3 + 9u_6^4, \label{deff2sextic}
\end{multline}
and the values of $w_i,\; i=1,\,2,\,\ldots,\,6$, in the simultaneous composition formulae \eqref{simulcompsexticuvw} are given by
\begin{equation}
\begin{aligned}
w_1 & = u_1v_1 + qu_2v_2,\\
 w_2 & = u_1v_2 + u_2v_1,\\
 w_3 & = u_1v_3 + qu_2v_4 + u_3v_1 - 2u_3v_3 - u_3v_6 + qu_4v_2 \\
& \quad - 2qu_4v_4 - qu_4v_5 - qu_5v_4 + qu_5v_5 - u_6v_3 + u_6v_6, \\
w_4 & = u_1v_4 + u_2v_6 - u_3v_4 + u_3v_5 + u_4v_1 - u_4v_3 - 2u_4v_6 + u_5v_3\\
& \quad - u_5v_6 + u_6v_2 - 2u_6v_4 - u_6v_5, \\
w_5 & = u_1v_5 + u_2v_3 + u_3v_2 - u_3v_4 - 2u_3v_5 - u_4v_3 \\
& \quad+ u_4v_6 + u_5v_1 - 2u_5v_3 - u_5v_6 + u_6v_4 - u_6v_5, \\
w_6 & = u_1v_6 + qu_2v_2 - qu_2v_4 - qu_2v_5 + u_3v_3 - u_3v_6 - qu_4v_2 + qu_4v_4\\
& \quad + 2qu_4v_5 - qu_5v_2 + 2qu_5v_4 + qu_5v_5 + u_6v_1 - u_6v_3 - 2u_6v_6.
\end{aligned}
\end{equation}

\subsubsection{}\label{sexticfms3}
With the forms $f_1(u_i)$ and $f_2(u_i)$ defined by \eqref{deff1sextic} and \eqref{deff2sextic} respectively, we will now consider the  the simultaneous diophantine equations $f_1(u_i)=1$ and $f_2(u_i)=1$ in the special case when $q=3$, that is, the simultaneous equations,
\begin{equation}
u_1^2 - 3u_2^2=1, \label{eq1simulsextic}
\end{equation}
and
\begin{multline}
u_1^4 - 6u_1^3u_3 - 6u_1^3u_6 + 3u_1^2u_2^2 - 18u_1^2u_2u_5 + 15u_1^2u_3^2\\
 + 24u_1^2u_3u_6 - 9u_1^2u_4^2 + 18u_1^2u_4u_5 + 18u_1^2u_5^2 + 15u_1^2u_6^2\\
 - 18u_1u_2^2u_6 - 36u_1u_2u_3u_4 + 36u_1u_2u_3u_5 + 36u_1u_2u_4u_6 \\
+ 72u_1u_2u_5u_6 - 18u_1u_3^3 - 36u_1u_3^2u_6 + 54u_1u_3u_4^2 - 54u_1u_3u_5^2\\
 - 36u_1u_3u_6^2 - 108u_1u_4u_5u_6 - 54u_1u_5^2u_6 - 18u_1u_6^3 + 9u_2^4\\
 - 54u_2^3u_4 - 54u_2^3u_5 - 9u_2^2u_3^2 + 18u_2^2u_3u_6 + 135u_2^2u_4^2\\
 + 216u_2^2u_4u_5 + 135u_2^2u_5^2 + 18u_2^2u_6^2 + 54u_2u_3^2u_4 \\
- 108u_2u_3u_5u_6 - 162u_2u_4^3 - 324u_2u_4^2u_5 - 324u_2u_4u_5^2 \\
- 54u_2u_4u_6^2 - 162u_2u_5^3 - 54u_2u_5u_6^2 + 9u_3^4 + 18u_3^3u_6 \\
- 54u_3^2u_4^2 - 54u_3^2u_4u_5 + 27u_3^2u_5^2 + 27u_3^2u_6^2 - 54u_3u_4^2u_6\\
 + 108u_3u_4u_5u_6 + 108u_3u_5^2u_6 + 18u_3u_6^3 + 81u_4^4 + 162u_4^3u_5\\
 + 243u_4^2u_5^2 + 27u_4^2u_6^2 + 162u_4u_5^3 + 108u_4u_5u_6^2 + 81u_5^4 \\
+ 27u_5^2u_6^2 + 9u_6^4=1, \label{eq2simulsextic}
\end{multline}
It is readily verified that a numerical solution of the simultaneous equations \eqref{eq1simulsextic} and \eqref{eq2simulsextic} is given by 
\begin{equation}(u_1, u_2, u_3, u_4, u_5, u_6)=(2, 1, 3, -1, 3, -4). \label{knownsolsimulsextic}
\end{equation}

By applying the composition identities \eqref{simulcompsexticuvw}, we can combine any integer solution $u_i=\alpha_{1i},\; i=1,\,2,\,\ldots,\,6$,  of the simultaneous diophantine equations \eqref{eq1simulsextic} and\eqref{eq2simulsextic}, with the known solution \eqref{knownsolsimulsextic} to obtain a new  solution $u_i=\alpha_{2i},\; i=1,\,2,\,\ldots,\,6$, of Eqs.~\eqref{eq1simulsextic} and \eqref{eq2simulsextic}. The new solution is given by
\begin{equation}
\begin{aligned}
\alpha_{21} & = 2\alpha_{11} + 3\alpha_{12}, \\
\alpha_{22}& = \alpha_{11} + 2\alpha_{12}, \\
\alpha_{23}& = 3\alpha_{11} - 3\alpha_{12} + 12\alpha_{15} - 7\alpha_{16}, \\
\alpha_{24}& = -\alpha_{11} - 4\alpha_{12} + 4\alpha_{13} + 7\alpha_{14} + 7\alpha_{15}, \\
\alpha_{25}& = 3\alpha_{11} + 3\alpha_{12} - 4\alpha_{13} - 7\alpha_{14} - 4\alpha_{16},\\
 \alpha_{26} & = -4\alpha_{11} - 3\alpha_{12} + 7\alpha_{13} + 12\alpha_{14} + 7\alpha_{16}.
\end{aligned}
\label{sol2simulsextic}
\end{equation}

If $\alpha_{11} > 0$ and $\alpha_{12} > 0$, it follows from \eqref{sol2simulsextic} that  $\alpha_{21} > \alpha_{11} > 0$ and $\alpha_{22} >\alpha_{12} >  0$, and hence the solution $u_i=\alpha_{2i},\; i=1,\,2,\,\ldots,\,6$, is distinct from the solution $u_i=\alpha_{1i},\; i=1,\,2,\,\ldots,\,6$. Further, we may now combine the solution  $u_i=\alpha_{2i},\; i=1,\,2,\,\ldots,\,6$, with the known solution \eqref{knownsolsimulsextic} to obtain  a new solution $u_i=\alpha_{3i},\; i=1,\,2,\,\ldots,\,6$, of Eqs.~\eqref{eq1simulsextic} and \eqref{eq2simulsextic} such that  $\alpha_{31} > \alpha_{21} > 0$ and $\alpha_{32} >\alpha_{22} >  0$, and we can repeat  the process any number of times  to obtain an infinite sequence of integer solutions of the simultaneous diophantine equations \eqref{eq1simulsextic} and \eqref{eq2simulsextic}.

If we take the initial solution 
$(\alpha_{11}, \alpha_{12}, \alpha_{13}, \alpha_{14}, \alpha_{15}, \alpha_{16})$ as $ (2, 1, 3, -1,$ $  3, -4)$, the next three solutions of the infinite sequence of integer solutions of Eqs.~\eqref{eq1simulsextic} and\eqref{eq2simulsextic} are given by 
\[
\begin{array}{c}
(7,\,  4,\,  67,\, 20,\, 20,\, -30),\, \quad (26,\,  15,\,  459,\,  525,\,  -255,\, 459),\,\\
( 97,\,  56,\,  -6240,\, 3640,\,  -7224,\, 12577).
\end{array}
\]

We know that a parametric solution of the Pell's equation does not exist. Since Eq.~\eqref{eq1simulsextic} is a Pell's equation, it cannot have a parametric solution. It follows that the infinite sequence of integer solutions of the simultaneous diophantine equations \eqref{eq1simulsextic} and \eqref{eq2simulsextic}  cannot be generated by a parametric solution or even by a finite set of parametric solutions. Further, since there can only be finitely many integer points on a curve of genus 1, the infinitely many integer solutions of Eq.~\eqref{eq2simulsextic} cannot be obtained from the integer points of any curve of genus 1. Thus, we have obtained infinitely many solutions in positive integers of the quartic equation \eqref{eq2simulsextic} and these solutions can neither be obtained from a parametric solution nor from  the integer points of a curve of genus 1. 

\subsection{Octic forms}\label{octicfms}
We will now construct octonary octic forms admitting composition and consider related octic diophantine equations.

\subsubsection{}\label{octicfms1} We follow the familiar procedure that has already been applied  in Sections~\ref{quarticfms1}, \ref{sexticfms1} and \ref{sexticfms2}.
We first choose the matrix $U$ as follows:
\begin{equation}
U=\begin{bmatrix}
a_1 & a_2 \\ -sa_2 & a_1+ra_2\end{bmatrix}. \label{valUoctic}
\end{equation}
This matrix has been obtained by a suitable renaming of the parameters in the matrix $U$ given by \eqref{valU}, and accordingly its linear structure is preserved under multiplication. We will now construct an $8 \times 8$  matrix $P$ by replacing the variables $a_1$ and $a_2$ in \eqref{valUoctic}  by two $4 \times 4$ matrices  $A_1$ and $A_2$ with  an identical linear structure that is preserved under multiplication. 

In Section~\ref{quarticfms1} we have already constructed the  $4 \times 4$ matrix \eqref{m44} whose linear structure is preserved under multiplication. We now name this matrix as $A(x_1,\,x_2,\,x_3,\,x_4)$, that is, we   write,
\small
\begin{equation*}
A(x_1,\ldots,x_4)=\begin{bmatrix}x_1 & x_2 & x_3 & x_4 \\ -n x_2 & x_1+m x_2 & -n x_4 & x_3+m x_4 \\ -q x_3 & -q x_4 & x_1+p x_3 & x_2+p x_4 \\ q n x_4 & -q (x_3+m x_4) & -nx_2-p n x_4 & x_1+m x_2+p (x_3+m x_4)
\end{bmatrix}
\end{equation*}
\normalsize
We now choose the matrices $A_1$ and $A_2$ as follows:
\begin{equation}
A_1=A(x_1,\,x_2,\,x_3,\,x_4),\quad A_2=A(x_5,\,x_6,\,x_7,\,x_8),
\end{equation}
and thus get the matrix $P(x_1,\,x_2,\,\ldots,\,x_8)$ as follows:
\begin{equation*}
P(x_1,\,x_2,\,\ldots,\,x_8)=\begin{bmatrix}
A(x_1,\,\ldots,\,x_4) & A(x_5,\,\ldots,\,x_8) \\ -sA(x_5,\,\ldots,\,x_8) & A(x_1,\,\ldots,\,x_4)+rA(x_5,\,\ldots,\,x_8)\end{bmatrix}. \label{valPoctic}
\end{equation*}

Since the linear structure of the matrix $A(x_1,\,x_2, x_3,\,x_4)$ is preserved under multiplication,  the linear structure of the matrix $P(x_1,\,x_2,\,\ldots,\,x_8)$ is also preserved under multiplication. Thus, on multiplying two such matrices, we get
\begin{equation}
P(x_1, x_2,\ldots, x_8)P( y_1,  y_2,\ldots,  y_8)=P( z_1,  z_2,  \ldots,  z_8), \label{octicmprod}
\end{equation}
where $x_i,\, y_i,\;i=1,\,\ldots,\,8$, are independent variables and the values of $z_i, \;i=1,\,\ldots,\,8$, are given by bilinear forms in the variables $x_i,\, y_i$.

 If we now write
\begin{equation}
f(x_1,\, x_2,\ldots,\,x_8)={\rm det}(\left(P(x_1,\,x_2,\,\dots,\,x_8)\right),\label{octicfm}
\end{equation}
we immediately get the composition formula,
\begin{equation}
f(x_1, x_2,\ldots, x_8)f( y_1,  y_2,\ldots,  y_8)=f( z_1,  z_2,  \ldots,  z_8), \label{complawocticfm}
\end{equation}
where $f(x_1, x_2,\ldots, x_8)$ is an octic form in the variables $x_i,\;i=1,\,\ldots,\,8$, while $m,\,n,\,p, \,q,\,r,\,s$,  are arbitrary integer parameters and the values of $z_i,\,i=1,\,\ldots,\,8$, are given by
\begin{equation}
\begin{aligned} 
z_1 &= x_1y_1 - nx_2y_2 - qx_3y_3 + qnx_4y_4 \\ & \quad \;\;
- sx_5y_5 + snx_6y_6 + sqx_7y_7 - sqnx_8y_8, \\
z_2 &= x_1y_2 + x_2y_1 + mx_2y_2 - qx_3y_4 - qx_4y_3 - qmx_4y_4\\ & \quad \;\;
 - sx_5y_6 - sx_6y_5 - smx_6y_6 + sqx_7y_8 + sqx_8y_7 + sqmx_8y_8, \\
z_3 &= x_1y_3 - nx_2y_4 + x_3y_1 + px_3y_3 - nx_4y_2 - npx_4y_4\\ & \quad \;\;
 - sx_5y_7 + snx_6y_8 - sx_7y_5 - spx_7y_7 + snx_8y_6 + snpx_8y_8, \\
z_4 &= x_1y_4 + x_2y_3 + mx_2y_4 + x_3y_2 + px_3y_4 + x_4y_1\\ & \quad \;\;
 + mx_4y_2 + px_4y_3 + pmx_4y_4 - sx_5y_8 - sx_6y_7 - smx_6y_8\\ & \quad \;\;
 - sx_7y_6 - spx_7y_8 - sx_8y_5 - smx_8y_6 - spx_8y_7 - spmx_8y_8, \\
z_5 &= x_1y_5 - nx_2y_6 - qx_3y_7 + qnx_4y_8 + x_5y_1 + rx_5y_5\\ & \quad \;\;
 - nx_6y_2 - nrx_6y_6 - qx_7y_3 - qrx_7y_7 + qnx_8y_4 + nqrx_8y_8, \\
z_6 &= x_1y_6 + x_2y_5 + mx_2y_6 - qx_3y_8 - qx_4y_7 - qmx_4y_8\\ & \quad \;\;
 + x_5y_2 + rx_5y_6 + x_6y_1 + mx_6y_2 + rx_6y_5 + rmx_6y_6\\ & \quad \;\;
 - qx_7y_4 - qrx_7y_8 - qx_8y_3 - qmx_8y_4 - qrx_8y_7 - qrmx_8y_8, \\
z_7 &= x_1y_7 - nx_2y_8 + x_3y_5 + px_3y_7 - nx_4y_6 - npx_4y_8\\ & \quad \;\;
 + x_5y_3 + rx_5y_7 - nx_6y_4 - nrx_6y_8 + x_7y_1 + px_7y_3\\ & \quad \;\;
 + rx_7y_5 + rpx_7y_7 - nx_8y_2 - npx_8y_4 - nrx_8y_6 - rnpx_8y_8,\\
 z_8 &= x_1y_8 + x_2y_7 + mx_2y_8 + x_3y_6 + px_3y_8 + x_4y_5\\ & \quad \;\;
 + mx_4y_6 + px_4y_7 + pmx_4y_8 + x_5y_4 + rx_5y_8 + x_6y_3 \\ & \quad \;\;
+ mx_6y_4 + rx_6y_7 + rmx_6y_8 + x_7y_2 + px_7y_4 + rx_7y_6\\ & \quad \;\;
 + rpx_7y_8 + x_8y_1 + mx_8y_2 + px_8y_3 + pmx_8y_4 + rx_8y_5\\ & \quad \;\;
 + rmx_8y_6 + rpx_8y_7 + rpmx_8y_8.
\end{aligned}
\label{valzoctic}
\end{equation}

It has been verified using MAPLE that the octic form $f(x_1, x_2,\ldots, x_8)$ is irreducible for arbitrary values of the parameters  $m,\,n,\,p, \,q,\,r,\,s$.  It is also  not the norm of an algebraic integer.

\subsubsection{}\label{octicfms2} We will now consider the  diophantine equation
\begin{equation}
f(x_1,\, x_2,\ldots,\,x_8)=1, \label{octiceqn}
\end{equation}
where $f(x_1,\, x_2,\ldots,\,x_8)$ is the octic form defined by \eqref{octicfm}. As in Sections~\ref{quarticfms2} and \ref{sexticfms1}, it is readily established that the set of integer solutions of the octic diophantine equation \eqref{octiceqn} forms an abelian  group, the identity being $(1,\,0,\,0,\,0,\,0,\,0,\,0,\,0)$. 

We will now consider Eq.~\eqref{octiceqn} when $(m,\,n,\,p,\,q,\,r,\,s)=( 0,\,  -5,\,  0,\,  -3,$  $  0,\, -14)$. It is readily verified that a numerical solution of this equation is given by
\begin{equation}(x_1, x_2, x_3, x_4, x_5, x_6, x_7, x_8)=(4,\,  2,\, 2,\, 1,\, 14,\, 7,\, 8,\, 4).\label{knownsoloctic}
\end{equation}

If $(\alpha_{11},\,\alpha_{12},\,\ldots,\,\alpha_{18})$ is an arbitrary  integer solution of  our numerical octic equation such that $\alpha_{1i} > 0$ for each $i$, we may use the composition identity \eqref{complawocticfm} and the solutions \eqref{knownsoloctic} and $(\alpha_{11},\,\ldots,\,\alpha_{18})$ to obtain a new solution $(\alpha_{21},\,\ldots,\,\alpha_{28})$ where
\begin{equation}
\begin{aligned}
 \alpha_{21}&=4\alpha_1 + 10\alpha_2 + 6\alpha_3 + 15\alpha_4 + 196\alpha_5 + 490\alpha_6 + 336\alpha_7 + 840\alpha_8, \\
 \alpha_{22}&=2\alpha_1 + 4\alpha_2 + 3\alpha_3 + 6\alpha_4 + 98\alpha_5 + 196\alpha_6 + 168\alpha_7 + 336\alpha_8, \\
 \alpha_{23}&= 2\alpha_1 + 5\alpha_2 + 4\alpha_3 + 10\alpha_4 + 112\alpha_5 + 280\alpha_6 + 196\alpha_7 + 490\alpha_8\\
 \alpha_{24}&=\alpha_1 + 2\alpha_2 + 2\alpha_3 + 4\alpha_4 + 56\alpha_5 + 112\alpha_6 + 98\alpha_7 + 196\alpha_8,\\
\alpha_{25}&= 14\alpha_1 + 35\alpha_2 + 24\alpha_3 + 60\alpha_4 + 4\alpha_5 + 10\alpha_6 + 6\alpha_7 + 15\alpha_8,\\
\alpha_{26}&= 7\alpha_1 + 14\alpha_2 + 12\alpha_3 + 24\alpha_4 + 2\alpha_5 + 4\alpha_6 + 3\alpha_7 + 6\alpha_8,\\
\alpha_{27}& = 8\alpha_1 + 20\alpha_2 + 14\alpha_3 + 35\alpha_4 + 2\alpha_5 + 5\alpha_6 + 4\alpha_7 + 10\alpha_8,\\
\alpha_{28}&= 4\alpha_1 + 8\alpha_2 + 7\alpha_3 + 14\alpha_4 + \alpha_5 + 2\alpha_6 + 2\alpha_7 + 4\alpha_8.
\end{aligned}
\label{valalpha2octiceqn}
\end{equation}

Since $\alpha_{1i} > 0$ for each $i$, it immediately follows from \eqref{valalpha2octiceqn} that $\alpha_{21} > \alpha_{11}$ and hence this solution is distinct from the solution $(\alpha_{11},\,\alpha_{12},\,\ldots,\,\alpha_{18})$. Further, $\alpha_{2i} > 0$ for each $i$, and we may  therefore use the solutions  \eqref{knownsoloctic} and $(\alpha_{21},\,\alpha_{22},\,\ldots,\,\alpha_{28})$ to get a  new solution $(\alpha_{31},\,\ldots,\,\alpha_{38})$ such that $\alpha_{31} > \alpha_{21} > \alpha_{11}$, and $\alpha_{3i} > 0$ for each $i$,  and we may repeat the process any number of times to get an infinite sequence of integer solutions of our octic equation.

If we  take the initial  solution $(\alpha_{11},\,\alpha_{12},\,\ldots,\,\alpha_{18})$ as $(4,  2, 2, 1, 14, 7, 8, 4)$, the next three solutions of our octic equation obtained by the above process are as follows:
\[
\begin{array}{c}
(12285,\,  5460,\,  7092,\,  3152,\,  468,\,  208,\,  270,\,  120),\,\\
(578740,\,  258910,\,  334134,\,  149481,\,  729790,\,  326485,\,  421344,\,  188496),\,\\
(612075793,\,  273723336,\,  353382120,\,  158034240,\,  45691800,\,  20433600,\, \\
\quad \quad  26380172,\,  11797344).
\end{array}\] 

\section{Three-fold composition of forms and related diophantine equations}\label{threefold}
We will now consider forms $f(x_1,\,x_2,\,\ldots,\,x_n)$ such that their exists an identity,
\begin{multline}
 f_1(x_1,\,x_2,\,\ldots,\,x_n)f_1(y_1,\,y_2,\,\ldots,\,y_n)f_1(z_1,\,z_2,\,\ldots,\,z_n)\\
=f_1(w_1,\,w_2,\,\ldots,\,w_n),\quad \quad \label{threefoldcomplaw}
\end{multline}
where  the variables $z_i,\;i=1,\,2,\,\ldots,\,n$, are given by trilinear forms in the variables $x_i,\,y_i,\,z_i,\;i=1,\,2,\,\ldots,\,n$.  If the form $f(x_1,\,x_2,\,\ldots,\,x_n)$ admits a composition identity \eqref{complaw}, by applying the identity \eqref{complaw}  twice, we can readily obtain an identity of type \eqref{threefoldcomplaw}. To avoid such obvious cases,  we say that a form $f(x_1,\,x_2,\,\ldots,\,x_n)$ admits three-fold composition if there exists an identity  \eqref{threefoldcomplaw} and further, such an identity  cannot be derived from an identity of type \eqref{complaw}.

We will restrict our attention to those forms $f(x_1,\,x_2,\,\ldots,\,x_n)$ admitting three-fold composition whose coefficients are integers and further the coefficients of the trilinear forms, associated with the composition identity \eqref{threefoldcomplaw}, are also integers.
   
In Section~\ref{threefoldquadfms} we show that every binary  quadratic form admits three-fold composition, in  Section~\ref{threefoldhighdegfms} we describe a general method of constructing higher degree forms which admit three-fold composition, and in the next two subsections we apply the general method to obtain quartic and octic forms which admit three-fold composition. We also discuss related diophantine equations in the respective subsections.

\subsection{Quadratic forms}\label{threefoldquadfms}

\begin{theorem}\label{Th1} For arbitrary integers $a,\,b,\,c$, the   binary quadratic form $Q(x_1,\,x_2)=ax_1^2 + bx_1x_2 + cx_2^2$ admits a three-fold composition law,
\begin{equation}
Q(x_1,\,x_2)Q(y_1,\,y_2)Q(z_1,\,z_2)=Q(w_1,\,w_2), \label{3foldqdident}
\end{equation}
where $x_1, x_2, y_1, y_2, z_1, z_2$ are arbitrary variables and, if we write
\begin{equation}
\begin{aligned}
\phi_1(x_1, x_2, y_1, y_2, z_1, z_2)&=ax_1y_1z_1 + bx_1y_2z_1 + cx_1y_2z_2 \\
& \quad \quad - cx_2y_1z_2 + cx_2y_2z_1,\\
\phi_2(x_1, x_2, y_1, y_2, z_1, z_2)&= ax_1y_1z_2 - ax_1y_2z_1 + ax_2y_1z_1\\
& \quad \quad  + bx_2y_1z_2 + cx_2y_2z_2,
\end{aligned}
\end{equation}
the values of $w_1, w_2$ are given either by,
\begin{align}
&w_1&=\phi_1(x_1, x_2, y_1, y_2, z_1, z_2),\;\; w_2&=\phi_2(x_1, x_2, y_1, y_2, z_1, z_2), \label{valwfirst}\\
{\it or \;\; by,\;\;}&w_1&=\phi_1(y_1, y_2, z_1, z_2, x_1, x_2),\;\; w_2&=\phi_2(y_1, y_2, z_1, z_2, x_1, x_2), \label{valwsec}\\
{\it or\;\;  by,\;\;}&w_1&=\phi_1(z_1, z_2, x_1, x_2, y_1, y_2),\;\; w_2&=\phi_2(z_1, z_2, x_1, x_2, y_1, y_2).\label{valwtre}
\end{align}
\end{theorem}

\begin{proof} While it is readily verified that the identity \eqref{3foldqdident} is true when the values of $w_1,\,w_2$ are given by \eqref{valwfirst} or \eqref{valwsec} or \eqref{valwtre}, we will derive the identity \eqref{3foldqdident} by using  matrices.  We consider the matrix $A(x_1, x_2)$ defined by
\begin{equation}
A(x_1, x_2)=\begin{bmatrix}
tx_1 & x_2 \\ bx_1 + cx_2 & -tx_1 
\end{bmatrix}.
\label{3fold22matrix}
\end{equation}
The entries of     $A(x_1, x_2)$ are linear forms in the variables $x_1,\,x_2$ while $b, c$ and $t$ are arbitrary. Thus the matrix  $A(x_1, x_2)$ has a certain linear structure and in particular, we note that ${\rm tr}(A(x_1, x_2))=0$. 

We note the trace of the matrix product $A(x_1,\,x_2)A(y_1,\,y_2)$ is not zero and hence  the linear structure of the matrix $A(x_1,\,x_2)$ is not preserved on multiplying two such matrices but when we multiply three such matrices, the linear structure is preserved as is seen from  the identity,
\begin{equation}
A(x_1,\,x_2)A(y_1,\,y_2)A(z_1,\,z_2)=A(w_1,\,w_2),\label{prodA3fold}
\end{equation}
where $w_1$ and $w_2$ are trilinear forms in the variables $x_i,\,y_i,\,z_i$ defined by
\begin{equation}
\begin{aligned}
w_1&=t^2x_1y_1z_1 + bx_1y_2z_1 + cx_1y_2z_2 - cx_2y_1z_2 + cx_2y_2z_1,  \\
w_2& = t^2x_1y_1z_2 - t^2x_1y_2z_1 + t^2x_2y_1z_1 + bx_2y_1z_2 + cx_2y_2z_2.
\end{aligned}
\label{valw3fold}
\end{equation}
It follows from \eqref{prodA3fold} that 
\begin{equation}
{\rm det}(A(x_1,\,x_2)) \times {\rm det}(A(y_1,\,y_2)) \times {\rm det}(A(z_1,\,z_2))={\rm det}(A(w_1,\,w_2)),
\end{equation}
and hence we get the identity
\begin{multline}
(t^2x_1^2 + bx_1x_2 + cx_2^2)(t^2y_1^2 + by_1y_2 + cy_2^2)(t^2z_1^2 + bz_1z_2 + cz_2^2)\\
 = t^2w_1^2 + bw_1w_2 + cw_2^2, \label{identqd3fold}
\end{multline}
where the values of $w_1,\,w_2$ are given by \eqref{valw3fold}. 

Since $t$ is arbitrary and it occurs only as $t^2$ in the relations \eqref{valw3fold} and \eqref{identqd3fold}, we may simply replace $t^2$ by $a$, and we then get the identity \eqref{3foldqdident} where the values of $w_1,\,w_2$ are given by  \eqref{valwfirst}. We note that on permuting the three pairs of variables $(x_1,\,x_2),\,(y_1,\,y_2),\,(z_1,\,z_2)$ in the identity just obtained, while  the left-hand side of the identity  remains unchanged,   the values of $w_1,\,w_2$ get modified, and accordingly  we get the additional values of $w_1,\,w_2$, given by \eqref{valwsec} and \eqref{valwtre}, for which also the identity \eqref{3foldqdident} is satisfied.

To prove that Theorem~\ref{Th1} actually gives a three-fold composition identity for the  form $Q(x_1,\,x_2)$, we  will now show that there cannot exist a composition identity $Q(x_1,\,x_2)Q(y_1,\,y_2)=Q(z_1,\,z_2)$. Since $a, b, c$, are arbitrary integers, we may choose them such that  $Q(x_1,\,x_2)$ is a negative definite form. The product $Q(x_1,\,x_2)Q(y_1,\,y_2)$ is thus necessarily positive, and hence cannot be expressed by the form $Q(z_1,\,z_2)$. Thus there cannot exist an identity $Q(x_1,\,x_2)Q(y_1,\,y_2)=Q(z_1,\,z_2)$ for arbitrary $a,\,b,\,c$, and  hence the composition identity \eqref{3foldqdident} is indeed a three-fold composition identity.
\end{proof}

Since the left-hand side of the identity \eqref{3foldqdident} remains the same for the three pairs of values of $w_1,\,w_2$ given by \eqref{valwfirst}, \eqref{valwsec} and \eqref{valwtre}, it immediately follows that a solution of the diophantine chain,
\begin{equation}
Q(u_1,\,u_2)=Q(v_1,\,v_2)=Q(w_1,\,w_2),
\end{equation}
is given by  trilinear forms in the variables $x_i,\,y_i,\,z_i$ by
\begin{equation}
\begin{aligned}
u_1&=\phi_1(x_1, x_2, y_1, y_2, z_1, z_2),\quad u_2&=\phi_2(x_1, x_2, y_1, y_2, z_1, z_2), \\
v_1&=\phi_1(y_1, y_2, z_1, z_2, x_1, x_2),\quad v_2&=\phi_2(y_1, y_2, z_1, z_2, x_1, x_2), \\
w_1&=\phi_1(z_1, z_2, x_1, x_2, y_1, y_2),\quad w_2&=\phi_2(z_1, z_2, x_1, x_2, y_1, y_2).
\end{aligned}
\end{equation}

\subsection{Higher degree forms admitting three-fold composition}\label{threefoldhighdegfms}

We will now show how to construct higher degree forms that admit three-fold composition by first constructing a matrix whose linear structure is preserved when three such matrices are multiplied  but not when only two such matrices are multiplied.  In Section~\ref{threefoldquadfms} we have already constructed a $2 \times 2$ matrix whose linear structure is preserved  when three such matrices are multiplied. We will now use a method analogous to the one described in Section~\ref{mmhom}  
to  construct a $mn \times mn $ matrix whose linear structure is preserved  when three such matrices are multiplied  but  not when 
only two such matrices are multiplied.

We will begin with a  matrix $U$ whose entries are linear forms in variables $a_i,\;i=1,\,2,\,\ldots,\,a_h$, and a set of matrices $A_i,\;i=1,\,2,\,\ldots,\,A_h$, all of which have an identical linear structure. As in Section~\ref{mmhom} we will construct a matrix $P$ of order $mn \times mn$ by replacing the variables $a_i,\;i=1,\,2,\,\ldots,\,a_h$ in the entries of the matrix $U$ by matrices $A_i,\;i=1,\,2,\,\ldots,\,A_h$ respectively. 

If the linear structure either of the matrix $U$ or of the matrices $A_i,\;i=1,\,2,\,\ldots,\,A_h$, is preserved  when three matrices with that  linear structure are multiplied but not when only two such matrices are multiplied, and the linear structure of the  remaining matrix or matrices is preserved when two such matrices are multiplied, the resulting matrix $P$ will have a linear structure that is preserved  when three matrices having the linear structure of the matrix $P$ are multiplied  but not when only two such matrices are multiplied. Similarly if the linear structure of the matrix  $U$ as well as the matrices  $A_i$ is preserved  when three matrices with the same  linear structure are multiplied but not when  two such matrices are multiplied, then also the linear structure of the matrix $P$ will be  preserved  when such three matrices  are multiplied  but not when only two such matrices are multiplied.

This may be established exactly as in Section~\ref{mmhom} and we accordingly omit further details.  The only additional point to be noted here is that since the linear structure either of the matrix $U$ or of the matrices $A_1,\,A_2,\,\ldots,\,A_h$ is not preserved when two such matrices are multiplied, the linear structure of the matrix $P$ is also not preserved when we multiply only  two  matrices with such a linear structure. 

If the entries of the matrix $P$ are written in terms of linear forms in the independent variables $x_1,\,x_2,\,\ldots,\,x_s$, we may write $P=P(x_1,\,x_2,\,\ldots,\,x_s)$, and we may multiply three matrices with the same linear structure to get a relation,
\begin{multline}
P(x_1,\,x_2,\,\ldots,\,x_s) \times P(y_1,\,y_2,\,\ldots,\,y_s) \times P(z_1,\,z_2,\,\ldots,\,z_s)\\
=P(w_1,\,w_2,\,\ldots,\,w_s), \label{3foldprodmatrix}
\end{multline}
where $x_i,\,y_i,\,z_i,\;i=1,\,2,\,\ldots,\,s$, are independent variables and $w_i,\;i=1,\,2,\,\ldots,\,s$, are trilinear forms in the variables $x_i,\,y_i,\,z_i,\;i=1,\,2,\,\ldots,\,s$. 

We now get a form of degree $mn$ by writing
\[
f(x_1,\,x_2,\,\ldots,\,x_s)={\rm det}(P(x_1,\,x_2,\,\ldots,\,x_s)),\]
and it immediately follows from \eqref{3foldprodmatrix} that the form $f(x_1,\,x_2,\,\ldots,\,x_s)$ satisfies a composition identity of type \eqref{threefoldcomplaw}.

In each specific example that we construct by this method, we would need to prove that the form $f(x_1,\,x_2,\,\ldots,\,x_s)$ does not satisfy a composition identity \eqref{complaw} in order to establish that we have obtained a form that genuinely admits three-fold composition. 

We will use the above method  in Sections~\ref{3foldquarticfms} and \ref{3foldocticfms} to obtain quartic and octic forms,  admitting three-fold composition, in 4 and 8 variables respectively.

\subsection{Quartic forms and a related quartic diophantine equation}\label{3foldquarticfms}

We will begin with the matrices,
\begin{equation}
U=\begin{bmatrix} ta_1  &   a_2\\ pa_1 + qa_2  &   -ta_1 \end{bmatrix}, \quad A(x_1,\,x_2)=\begin{bmatrix} sx_1  &   x_2\\ mx_1 + nx_2  &   -sx_1 \end{bmatrix},
\label{valUA}
\end{equation}
and follow the procedure described in Section~\ref{threefoldhighdegfms} to construct a $4 \times 4$ matrix $P(x_1,\,x_2,\,x_3,\,x_4)$    whose linear structure will be preserved when three matrices with the same linear structure are multiplied.

We note that both the matrices $U$ and $A(x_1,\,x_2)$ are obtained by a suitable renaming of the variables in the matrix \eqref{3fold22matrix} and accordingly, they have a   linear structure which  will be preserved when three  matrices with the same linear structure are multiplied. We now construct the matrix $P(x_1,\,x_2,\,x_3,\,x_4)$ by simply replacing the variables $a_1,\,a_2$ in the matrix $U$ by the matrices $A(x_1,\,x_2)$ and $A(x_3,\,x_4)$ respectively, and we get,
\begin{equation}
P(x_1,\,x_2,\,x_3,\,x_4)=\begin{bmatrix} tA(x_1,\,x_2)  &   A(x_3,\,x_4) \\ pA(x_1,\,x_2) + qA(x_3,\,x_4)  &   -tA(x_1,\,x_2) \end{bmatrix}. 
\end{equation}

The linear structure of the matrix $P(x_1,\,x_2,\,x_3,\,x_4)$ is preserved when three such matrices are multiplied and, as discussed in Section~\ref{threefoldhighdegfms}, we  readily obtain the identity,
\begin{equation}
f(x_1,\,\ldots,\,x_4)f(y_1,\,\ldots,\,y_4)f(z_1,\,\ldots,\,z_4)=f(w_1,\,\ldots,\,w_4), \label{3foldcompidentquarticfm}
\end{equation}
where \begin{multline}
f(x_1,\,\ldots,\,x_4)=s^4t^4x_1^4 + 2s^2t^4mx_1^3x_2 + 2s^4t^2px_1^3x_3\\
 + s^2t^2mpx_1^3x_4 + (m^2 + 2s^2n)t^4x_1^2x_2^2 + 3s^2t^2mpx_1^2x_2x_3\\
 + (m^2 + 2s^2n)t^2px_1^2x_2x_4 + (p^2 + 2t^2q)s^4x_1^2x_3^2 \\
+ (p^2 + 2t^2q)s^2mx_1^2x_3x_4 + (s^2np^2 + t^2m^2q - 2s^2t^2nq)x_1^2x_4^2\\
 + 2t^4mnx_1x_2^3 + (m^2 + 2s^2n)t^2px_1x_2^2x_3 + 3t^2mnpx_1x_2^2x_4\\
 + (p^2 + 2t^2q)s^2mx_1x_2x_3^2 + (m^2p^2 + 8s^2t^2nq)x_1x_2x_3x_4\\
 + (p^2 + 2t^2q)mnx_1x_2x_4^2 + 2s^4pqx_1x_3^3 + 3s^2mpqx_1x_3^2x_4\\
 + (m^2 + 2s^2n)pqx_1x_3x_4^2 + mnpqx_1x_4^3 + t^4n^2x_2^4
+ t^2mnpx_2^3x_3\\
 + 2t^2n^2px_2^3x_4 + (s^2np^2 + t^2m^2q - 2s^2t^2nq)x_2^2x_3^2\\
 + (p^2 + 2t^2q)mnx_2^2x_3x_4 + (p^2 + 2t^2q)n^2x_2^2x_4^2 + s^2mpqx_2x_3^3\\
 + (m^2 + 2s^2n)pqx_2x_3^2x_4 + 3mnpqx_2x_3x_4^2 + 2n^2pqx_2x_4^3\\
 + s^4q^2x_3^4 + 2s^2mq^2x_3^3x_4 + (m^2 + 2s^2n)q^2x_3^2x_4^2 + 2mnq^2x_3x_4^3 + n^2q^2x_4^4, \label{3foldquarticfm}
\end{multline}
and $x_i,\,y_i,\,z_i,\;i=1,\,\ldots,\,4$, are independent variables while the values of $w_i,\;i=1,\,\ldots,\,4$, are given by
\begin{equation}
\begin{aligned}
w_1   & =  s^2t^2x_1y_1z_1 + mt^2x_1y_2z_1 + nt^2x_1y_2z_2 - nt^2x_2y_1z_2\\ & \quad \;\;
 + nt^2x_2y_2z_1 + ps^2x_1y_3z_1 + qs^2x_1y_3z_3 - qs^2x_3y_1z_3 \\ & \quad \;\;
+ qs^2x_3y_3z_1 + mpx_1y_4z_1 + mqx_1y_4z_3 - mqx_3y_2z_3 \\ & \quad \;\;
+ mqx_3y_4z_1+ npx_1y_4z_2 - npx_2y_3z_2 + npx_2y_4z_1 \\ & \quad \;\;
+ nqx_1y_4z_4 - nqx_2y_3z_4 + nqx_2y_4z_3 - nqx_3y_2z_4 \\ & \quad \;\;
+ nqx_3y_4z_2 + nqx_4y_1z_4 - nqx_4y_2z_3- nqx_4y_3z_2\\ & \quad \;\; + nqx_4y_4z_1,\\
w_2   & =  s^2t^2x_1y_1z_2 - s^2t^2x_1y_2z_1 + s^2t^2x_2y_1z_1 + mt^2x_2y_1z_2\\ & \quad \;\; 
+ nt^2x_2y_2z_2 + ps^2x_1y_3z_2 - ps^2x_1y_4z_1 + ps^2x_2y_3z_1 \\ & \quad \;\;
+ qs^2x_1y_3z_4 - qs^2x_1y_4z_3 + qs^2x_2y_3z_3 - qs^2x_3y_1z_4 \\ & \quad \;\;
+ qs^2x_3y_2z_3 + qs^2x_3y_3z_2 - qs^2x_3y_4z_1 - qs^2x_4y_1z_3 \\ & \quad \;\;
+ qs^2x_4y_3z_1 + mpx_2y_3z_2 + mqx_2y_3z_4 - mqx_4y_1z_4 \\ & \quad \;\;
+ mqx_4y_3z_2 + npx_2y_4z_2 + nqx_2y_4z_4 - nqx_4y_2z_4 \\ & \quad \;\;+ nqx_4y_4z_2, \\
w_3   & =  s^2t^2x_1y_1z_3 - s^2t^2x_1y_3z_1 + s^2t^2x_3y_1z_1 + mt^2x_1y_2z_3\\ & \quad \;\;
 - mt^2x_1y_4z_1 + mt^2x_3y_2z_1 + nt^2x_1y_2z_4 - nt^2x_1y_4z_2 \\ & \quad \;\;
- nt^2x_2y_1z_4 + nt^2x_2y_2z_3 + nt^2x_2y_3z_2 - nt^2x_2y_4z_1\\ & \quad \;\;
 + nt^2x_3y_2z_2 - nt^2x_4y_1z_2 + nt^2x_4y_2z_1 + ps^2x_3y_1z_3 \\ & \quad \;\;
+ qs^2x_3y_3z_3 + mpx_3y_2z_3 + mqx_3y_4z_3 + npx_3y_2z_4 \\ & \quad \;\;
- npx_4y_1z_4 + npx_4y_2z_3 + nqx_3y_4z_4 - nqx_4y_3z_4 \\ & \quad \;\;+ nqx_4y_4z_3,\\
 w_4   & =  s^2t^2x_1y_1z_4 - s^2t^2x_1y_2z_3 - s^2t^2x_1y_3z_2 + s^2t^2x_1y_4z_1\\ & \quad \;\;
 + s^2t^2x_2y_1z_3 - s^2t^2x_2y_3z_1 + s^2t^2x_3y_1z_2 - s^2t^2x_3y_2z_1 \\ & \quad \;\;
+ s^2t^2x_4y_1z_1 + mt^2x_2y_1z_4 - mt^2x_2y_3z_2 + mt^2x_4y_1z_2 \\ & \quad \;\;
+ nt^2x_2y_2z_4 - nt^2x_2y_4z_2 + nt^2x_4y_2z_2 + ps^2x_3y_1z_4 \\ & \quad \;\;
- ps^2x_3y_2z_3 + ps^2x_4y_1z_3 + qs^2x_3y_3z_4 - qs^2x_3y_4z_3 \\ & \quad \;\;
+ qs^2x_4y_3z_3 + mpx_4y_1z_4 + mqx_4y_3z_4 + npx_4y_2z_4 \\ & \quad \;\;+ nqx_4y_4z_4.
\end{aligned}
\label{3foldvalwquarticfm}
\end{equation}

We will now show that the form $f(x_1,\,\ldots,\,x_4)$ defined by \eqref{3foldquarticfm} does not satisfy any composition identity of type \eqref{complaw}. Assuming such an identity exists, it would be valid for all values of the integer parameters $m, \, n,\,p,\,q,\, s,\,t$. We now choose 
$(m,\,n,\,p,\,q,\,s,\,t)=(0,\,1,\,0,\,2,\,0,\,0)$ when \eqref{complaw} reduces to $(4x_4^4)(4y_4^4)=4z_4^4$ which is false  since the value of $z_4$ must be  given by a bilinear form  with integer coefficients. It follows that the form $f(x_1,\,\ldots,\,x_4)$ does not satisfy any identity of type \eqref{complaw}. Hence it is indeed a form admitting three-fold composition. 

We will now consider the quartic diophantine equation $f(x_1,\,\ldots,\,x_4)=1$ where $f(x_1,\,\ldots,\,x_4)$ is defined by \eqref{3foldquarticfm} and 
$ (m,\,n,\,p,\,q,\,s,\,t)=(-1,\,-4,\,1,$ $-1,\,1,\,1)$, that is, the equation,
\begin{multline}
x_1^4 - 2x_1^3x_2 + 2x_1^3x_3 - x_1^3x_4 - 7x_1^2x_2^2 - 3x_1^2x_2x_3\\
 - 7x_1^2x_2x_4 - x_1^2x_3^2 + x_1^2x_3x_4 - 13x_1^2x_4^2 + 8x_1x_2^3\\
 - 7x_1x_2^2x_3 + 12x_1x_2^2x_4 + x_1x_2x_3^2 + 33x_1x_2x_3x_4\\
 - 4x_1x_2x_4^2 - 2x_1x_3^3 + 3x_1x_3^2x_4 + 7x_1x_3x_4^2 - 4x_1x_4^3\\
 + 16x_2^4 + 4x_2^3x_3 + 32x_2^3x_4 - 13x_2^2x_3^2 - 4x_2^2x_3x_4 \\
- 16x_2^2x_4^2 + x_2x_3^3 + 7x_2x_3^2x_4 - 12x_2x_3x_4^2 - 32x_2x_4^3 \\
+ x_3^4 - 2x_3^3x_4 - 7x_3^2x_4^2 + 8x_3x_4^3 + 16x_4^4=1. \label{3foldquarticeqn}
\end{multline}

It is readily verified that two numerical solutions of Eq.~\eqref{3foldquarticeqn} are  given by
\[
(x_1,\,x_2,\,x_3,\,x_4) =(1,\,0,\,0,\,0),\quad (x_1,\,x_2,\,x_3,\,x_4) =(21,\, 8, \,33,\,  13).\] 

 If $(\alpha_{11},\,\alpha_{12},\,\alpha_{13},\,\alpha_{14})$ is an arbitrary  integer solution of Eq.~\eqref{3foldquarticeqn} such that $\alpha_{1i} > 0$ for each $i$, in the identity \eqref{3foldcompidentquarticfm} we take,
\begin{equation*}
\begin{aligned}
(x_1,\,x_2,\,x_3,\,x_4)&=(\alpha_{11},\,\alpha_{12},\,\alpha_{13},\,\alpha_{14}),\;\;(y_1,\,y_2,\,y_3,\,y_4)&=(1,\,0,\,0,\,0),\\
(z_1,\,z_2,\,z_3,\,z_4)&=(21,\, 8, \,33,\,  13), &
\end{aligned}
\end{equation*}
and, on using the relations \eqref{3foldvalwquarticfm}, we obtain a  new solution $(\alpha_{21},\,\alpha_{22},$ $\alpha_{23},\,\alpha_{24})$ of Eq.~\eqref{3foldquarticeqn} where
\begin{equation}
\begin{aligned}
 \alpha_{21}&=  21\alpha_{11 } + 32\alpha_{12 } + 33\alpha_{13 } + 52\alpha_{14}, \\
 \alpha_{22}&=  8\alpha_{11 } + 13\alpha_{12 } + 13\alpha_{13 } + 20\alpha_{14}, \\
 \alpha_{23}&=33\alpha_{11 } + 52\alpha_{12 } + 54\alpha_{13 } + 84\alpha_{14}, \\
 \alpha_{24}&=13\alpha_{11 } + 20\alpha_{12 } + 21\alpha_{13 } + 33\alpha_{14}.
\end{aligned}
\end{equation}

Since $\alpha_{1i} > 0$ for each $i$, it immediately follows that $\alpha_{21} > \alpha_{11}$ and hence this solution is distinct from the solution $(\alpha_{11},\,\alpha_{12},\,\alpha_{13},\,\alpha_{14})$. Further, $\alpha_{2i} > 0$ for each $i$, and  therefore  we may now take $(x_1,\,x_2,\,x_3,\,x_4)=(\alpha_{21},\,\alpha_{22},\,\alpha_{23},\,\alpha_{24})$ and repeat  the above process to get a new solution $(\alpha_{31},\,\alpha_{32},\,\alpha_{33},\,\alpha_{34})$ such that $\alpha_{31} > \alpha_{21} > \alpha_{11}$, and $\alpha_{3i} > 0$ for each $i$,  and, in fact,  we may repeat the process any number of times to get an infinite sequence of integer solutions of Eq.~\eqref{3foldquarticeqn}.

If we  take the initial  solution $(\alpha_{11},\,\alpha_{12},\,\alpha_{13},\,\alpha_{14})$ as $(21,\, 8, \,33,\,  13)$, the next three solutions of Eq.~\eqref{3foldquarticeqn} obtained by the above process are as follows:
\[
\begin{array}{c}
( 2462, \, 961,\, 3983, \, 1555),\quad ( 294753, \, 115068,\, 476920,\, 186184),\\
( 35291917,\,  13777548,\, 57103521,\,  22292541).
\end{array}\]

\subsection{Octic forms and a related octic diophantine equation}\label{3foldocticfms}

We will now construct an octonary octic form admitting three-fold composition beginning  with the matrix $U$  defined by
\begin{equation}
U=\begin{bmatrix}
a_1 & a_2 \\ -sa_2 & a_1+ra_2\end{bmatrix}. \label{3foldvalUoctic}
\end{equation}
This matrix has been obtained by a suitable renaming of the variables in the matrix $U$ defined by \eqref{valU} and its linear structure is accordingly preserved when two such matrices are multiplied.

We will now construct an $8 \times 8$  matrix $P$ by replacing the variables $a_1$ and $a_2$ in \eqref{3foldvalUoctic}  by two $4 \times 4$ matrices  $A_1$ and $A_2$ with  an identical linear structure that is preserved under when three such matrices are multiplied.

To construct a $4 \times 4$ matrix whose linear structure  is preserved  when three such matrices are multiplied, we begin with the following two matrices:
\begin{equation}
M_1=\begin{bmatrix}
a_1 & a_2 \\ -qa_2 & a_1+pa_2\end{bmatrix}, \quad M_2(x_1, x_2)=\begin{bmatrix}
tx_1 & x_2 \\ mx_1 + nx_2 & -tx_1
\end{bmatrix}.  \label{3foldvalM12}
\end{equation}

The matrix $M_1$  has been obtained by simply renaming the matrix defined by \eqref{valU} while the matrix $M_2$ has been obtained by suitably renaming the parameters in the  matrix defined by \eqref{3fold22matrix}, and accordingly the linear structure of the matrix $M_1$ is preserved when two such matrices are multiplied while that of  
the matrix $M_2(x_1, x_2)$ is preserved when three such matrices are multiplied. 

Now on replacing the variables $a_1,\,a_2$ in the matrix $M_1$ by matrices $M_2(x_1, x_2)$ and $M_2(x_3, x_4)$, we get the following  $4 \times 4$ matrix $A(x_1, x_2, x_3, x_4)$ whose linear structure is preserved when such three matrices  are multiplied: 

\scriptsize
\begin{equation*}
A(x_1,\,x_2,\,x_3,\,x_4)=\begin{bmatrix}
t x_1 &x_2 &t x_3 &x_4\\ mx_1 +nx_2 &-t x_1 &m x_3+n x_4 &-t x_3\\ -q t x_3 &-q x_4 &t x_1+p t x_3 &x_2+p x_4\\ -q (m x_3+n x_4) &q t x_3 &mx_1 +nx_2 +p (m x_3+n x_4) &-t x_1-p t x_3
\end{bmatrix}. 
\end{equation*}
\normalsize

Now on replacing the variables $a_1,\,a_2$ in the matrix $U$ by matrices $A(x_1,\,x_2,\,x_3,\,x_4)$ and $A(x_5, x_6,\,x_7,\,x_8)$ respectively, we get an $8 \times 8$ matrix $P(x_1,\ldots, x_8)$ whose entries are linear forms in the variables $x_1, \ldots,\, x_8$ and whose  linear structure is preserved when three matrices
with  the same linear are multiplied. The matrix $P(x_1,\ldots, x_8)$ may be written as 
\begin{equation*}
P(x_1,\ldots,x_8)=\begin{bmatrix}
A(x_1,\ldots,x_4) & A(x_5,\ldots,x_8) \\ -sA(x_5,\ldots,x_8) & A(x_1,\ldots,x_4)+rA(x_5,\ldots,x_8)
\end{bmatrix}. \label{3foldmm88}
\end{equation*}

Since the linear structure of the matrix $P(x_1,\,\ldots,\, x_8)$ is preserved when we multiply three such matrices, 
it follows that if we write, 
\begin{equation}
f(x_1,\,\ldots,\,x_8)={\rm det}(P(x_1,\,\ldots,\,x_8)), \label{threefoldocticfm} 
\end{equation}
we have the identity,
\begin{equation}
f(x_1,\,\dots,\,x_8)f(y_1,\,\ldots,\,y_8)f(z_1,\,\ldots,\,z_8)=f(w_1,\,\ldots,\,w_8), \label{3foldidentocticfm}
\end{equation}
where the values of $w_i,\;i=1,\,\ldots,\,8$, are given by trilinear forms in the variables $x_i,\,y_i,\,z_i,\;i=1,\ldots,\,8$.
The values of $w_i$ are too cumbersome to write and are accordingly not being given explicitly.

We will now show that the form $f(x_1,\,\ldots,\,x_8)$  does not satisfy any composition identity of type \eqref{complaw}. Assuming such an identity exists, it would be valid for all values of the  parameters $m, \, n,\,p,\,q,\,r,\, s,\,t$. We choose 
$(m,\,n,\,p,\,q,\,r,\,s,\,t)=(0,\,2,\,0,\,0,\,0,\,0)$ when \eqref{complaw} reduces to $(16x_2^8)(16y_2^8)=16z_2^8$ which is false since the value of $z_2$ must be  given by a bilinear form  with integer coefficients. It follows that the form $f(x_1,\,\ldots,\,x_8)$ does not satisfy any identity of type \eqref{complaw}. Hence it is indeed a form admitting three-fold composition.

We now consider the octic  diophantine equation
\begin{equation}
f(x_1,\,\dots,\,x_8)=1. \label{3foldocticeqn}
\end{equation}
where $f(x_1,\,\dots,\,x_8)$ is defined by \eqref{threefoldocticfm} and 
$(m,\,n,\,p,\,q,\,r,\,s,\,t)=(3,\,-1,$ $ 0,\, -3,\,0,\,-14,\,1)$, so that the matrix 
$P(x_1,\,\ldots,\,x_8)$ may be written as,
\scriptsize
\begin{equation*}
P=\begin{bmatrix} x_1  &  x_2  &  x_3  &  x_4  &  x_5  &  x_6  &  x_7  &  x_8 \\  3 x_1-x_2  &  -x_1  &  3 x_3-x_4  &  -x_3  &  3 x_5-x_6  &  -x_5  &  3 x_7-x_8  &  -x_7 \\  3 x_3  &  3 x_4  &  x_1  &  x_2  &  3 x_7  &  3 x_8  &  x_5  &  x_6 \\  9 x_3-3 x_4  &  -3 x_3  &  3 x_1-x_2  &  -x_1  &  9 x_7-3 x_8  &  -3 x_7  &  3 x_5-x_6  &  -x_5 \\  14 x_5  &  14 x_6  &  14 x_7  &  14 x_8  &  x_1  &  x_2  &  x_3  &  x_4 \\  42 x_5-14 x_6  &  -14 x_5  &  42 x_7-14 x_8  &  -14 x_7  &  3 x_1-x_2  &  -x_1  &  3 x_3-x_4  &  -x_3 \\  42 x_7  &  42 x_8  &  14 x_5  &  14 x_6  &  3 x_3  &  3 x_4  &  x_1  &  x_2 \\  126 x_7-42 x_8  &  -42 x_7  &  42 x_5-14 x_6  &  -14 x_5  &  9 x_3-3 x_4  &  -3 x_3  &  3 x_1-x_2  &  -x_1
\end{bmatrix},
\label{3foldmm88ex1}
\end{equation*}
 \normalsize 
and now the octic equation is given by
\begin{equation}
{\rm det}(P)=1. \label{3foldocticeqnex1}
\end{equation}
We will show that Eq.~\eqref{3foldocticeqnex1} has infinitely many solutions in positive integers.

It follows from \eqref{3foldidentocticfm} that any three solutions of Eq.~\eqref{3foldocticeqnex1} may be combined     to yield a new solution. Now it is readily verified that two numerical solutions of Eq.~\eqref{3foldocticeqnex1} are $(1,0,0,$ $ 0,0,0,0,0)$ and $(2, 6,  1,  3,  7, 21,  4, 12)$. If $(\alpha_{11},\ldots,\,\alpha_{18})$ is an arbitrary solution of Eq.~\eqref{3foldocticeqnex1} such that $\alpha_{1i} > 0$ for each $i$, we may combine the three solutions $(\alpha_{11},\ldots,\,\alpha_{18})$, $(1,0,0,0,0,0,0,0)$ and $(2, 6,  1,  3,  7, 21,  4, 12)$, taken in that order, to get a new solution $(\alpha_{21},\ldots,\alpha_{28})$ which is given by
\begin{equation*}
\begin{aligned}
\alpha_{21}  & = 2\alpha_{11}  + 6\alpha_{12}  + 3\alpha_{13}  + 9\alpha_{14}  + 98\alpha_{15}  + 294\alpha_{16}  + 168\alpha_{17}  + 504\alpha_{18},\\
 \alpha_{22}  & = 6\alpha_{11}  + 20\alpha_{12}  + 9\alpha_{13}  + 30\alpha_{14}  + 294\alpha_{15}  + 980\alpha_{16}  + 504\alpha_{17}  + 1680\alpha_{18}, \\
\alpha_{23}  & = \alpha_{11}  + 3\alpha_{12}  + 2\alpha_{13}  + 6\alpha_{14}  + 56\alpha_{15}  + 168\alpha_{16}  + 98\alpha_{17}  + 294\alpha_{18}, \\
\alpha_{24}  & = 3\alpha_{11}  + 10\alpha_{12}  + 6\alpha_{13}  + 20\alpha_{14}  + 168\alpha_{15}  + 560\alpha_{16}  + 294\alpha_{17}  + 980\alpha_{18},\\
 \alpha_{25 }  & = 7\alpha_{11}  + 21\alpha_{12}  + 12\alpha_{13}  + 36\alpha_{14}  + 2\alpha_{15}  + 6\alpha_{16}  + 3\alpha_{17}  + 9\alpha_{18},\\
 \alpha_{26}  & = 21\alpha_{11}  + 70\alpha_{12}  + 36\alpha_{13}  + 120\alpha_{14}  + 6\alpha_{15}  + 20\alpha_{16}  + 9\alpha_{17}  + 30\alpha_{18},\\
 \alpha_{27}  & = 4\alpha_{11}  + 12\alpha_{12}  + 7\alpha_{13}  + 21\alpha_{14}  + \alpha_{15}  + 3\alpha_{16}  + 2\alpha_{17}  + 6\alpha_{18}, \\
\alpha_{28}  & = 12\alpha_{11}  + 40\alpha_{12}  + 21\alpha_{13}  + 70\alpha_{14}  + 3\alpha_{15}  + 10\alpha_{16}  + 6\alpha_{17}  + 20\alpha_{18}.
\end{aligned}
\label{sol23foldocticeqnex1}
\end{equation*}

Since $\alpha_{1i} > 0$ for each $i$, it immediately follows that $\alpha_{21} > \alpha_{11}$ and hence this solution is distinct from the solution $(\alpha_{11},\,\ldots,\,\alpha_{18})$. Further, $\alpha_{2i} > 0$ for each $i$, and we may  therefore  repeat the above process to get a new solution $(\alpha_{31},\,\ldots,\,\alpha_{38})$ such that $\alpha_{31} > \alpha_{21} > \alpha_{11}$, and $\alpha_{3i} > 0$ for each $i$,  and, in fact,  we may repeat the process any number of times to get an infinite sequence of integer solutions of Eq.~\eqref{3foldocticeqnex1}.

If we  take the initial  solution $(\alpha_{11},\,\dots,\,\alpha_{18})$ as $(2, 6,  1,  3,  7, 21,  4, 12)$, the next two solutions of Eq.~\eqref{3foldocticeqnex1} obtained by the above process are as follows:
\[
\begin{array}{c}
(13650,\,  45045,\,  7880,\,  26004,\,  520,\, 1716,\,  300,\,  990),\,\\
(1660070,\,  5482800,\, 958437,\,  3165480,\,  2093345,\,  6913800,\,  1208592,\, 3991680),\\
( 4520236757,\,  14929326951,\, 2609759880,\,  8619450840,\,\\
  337438200,\,  1114482600,\,  194820028,\,  643446804).
\end{array}\]
 
\section{Concluding remarks}
The composable forms constructed in Sections~\ref{highdegfms} and \ref{threefold} above are  illustrative examples, and many more forms admitting composition may be obtained in a similar manner. In fact, the  general methods given in this paper  may be used to construct forms of arbitrarily high degree admitting  the composition identity \eqref{complaw}  or the three-fold composition identity \eqref{f3compform}, and such that the forms constructed as well as the bilinear/ trilinear forms related to these identities have only integer coefficients. 

It would be of interest to find matrices with a linear structure that is preserved when $m$ matrices with the same linear structure are multiplied but not preserved when fewer than $m$ such matrices are multiplied where  $m > 3$. This would lead to a natural generalization of the idea of three-fold composition of forms to  $m$-fold composition of forms. 

As regards the diophantine equations, we have shown that the infinitely many solutions of the quartic diophantine equation \eqref{deff2sextic} cannot be obtained by  any parametric solution. It would be of interest to establish that the infinitely many solutions of the other quartic and octic equations solved in Sections~\ref{highdegfms} and \ref{threefold} also cannot be obtained by parametric solutions. 

The examples of diophantine equations given  in Sections~\ref{highdegfms} and \ref{threefold} are also only illustrative in nature. It is clear that many more similar examples can readily be constructed. In fact, using the methods described in this paper, it may  be possible to construct diophantine equations $f(x_i)=1$ with infinitely many integer solutions where $f(x_i)$ is a form of degree $n$ in $n$ variables and $n >  8$. 

\begin{center}
\Large
Acknowledgments
\end{center}

I wish to  thank the Harish-Chandra Research Institute, Prayagraj for providing me with all necessary facilities that have helped me to pursue my research work in mathematics.

\medskip

\noindent Postal Address: Ajai Choudhry, 
\newline \hspace{1.05 in}
13/4 A Clay Square,
\newline \hspace{1.05 in} Lucknow - 226001, INDIA.
\newline \noindent  E-mail: ajaic203@yahoo.com
\end{document}